\documentclass[english,12pt]{smfart}
\usepackage{amssymb}
\usepackage{amsfonts}
\usepackage{amscd}
\usepackage[all]{xypic}

\CompileMatrices

\swapnumbers

\def\ord{{\rm ord}}
\def\ac{{\overline{\rm ac}}}

\def\GL{{\rm GL}}

\let\cal\mathcal

\def\11{{\mathbf 1}}
\def\AA{{\mathbf A}}

\def\CC{{\mathbf C}}

\def\LL{{\mathbf L}}

\def\NN{{\mathbf N}}

\def\PP{{\mathbf P}}
\def\QQ{{\mathbf Q}}
\def\RR{{\mathbf R}}

\def\ZZ{{\mathbf Z}}

\def\cC{{\mathcal C}}
\def\cD{{\mathcal D}}

\def\cL{{\mathcal L}}
\def\cM{{\mathcal M}}

\def\cV{{\mathcal V}}

\mathchardef\alphag="7C0B
\mathchardef\betag="7C0C
\mathchardef\gammag="7C0D
\mathchardef\deltag="7C0E
\mathchardef\varepsilong="7C22
\mathchardef\varphig="7C27
\mathchardef\psig="7C20
\mathchardef\zetag="7C10
\mathchardef\epsilong="7C0F
\mathchardef\rhog="7C1A
\mathchardef\taug="7C1C
\mathchardef\upsilong="7C1D
\mathchardef\iotag="7C13
\mathchardef\thetag="7C12
\mathchardef\pig="7C19
\mathchardef\sigmag="7C1B
\mathchardef\etag="7C11
\mathchardef\omegag="7C21
\mathchardef\kappag="7C14
\mathchardef\lambdag="7C15
\mathchardef\mug="7C16
\mathchardef\xig="7C18
\mathchardef\chig="7C1F
\mathchardef\nug="7C17
\mathchardef\varthetag="7C23
\mathchardef\varpig="7C24
\mathchardef\varrhog="7C25
\mathchardef\varsigmag="7C26
\mathchardef\Omegag="7C0A
\mathchardef\Thetag="7C02
\mathchardef\Sigmag="7C06
\mathchardef\Deltag="7C01
\mathchardef\Phig="7C08
\mathchardef\Gammag="7C00
\mathchardef\Psig="7C09
\mathchardef\Lambdag="7C03
\mathchardef\Xig="7C04
\mathchardef\Pig="7C05
\mathchardef\Upsilong="7C07

\newtheorem{theorem}[subsubsection]{Theorem}
\newtheorem{lem}[subsubsection]{Lemma}
\newtheorem{cor}[subsubsection]{Corollary}
\newtheorem{prop}[subsubsection]{Proposition}

\newtheorem{claim}[subsubsection]{Claim}

\theoremstyle{definition}
\newtheorem{definition}[subsubsection]{Definition}
\newtheorem{example}[subsubsection]{Example}

\newtheorem{def-prop}[subsubsection]{Proposition-Definition}
\newtheorem{def-theorem}[subsubsection]{Theorem-Definition}
\newtheorem{def-lem}[subsubsection]{Lemma-Definition}

\theoremstyle{remark}
\newtheorem{remark}[subsubsection]{Remark}

\theoremstyle{remark}

\theoremstyle{plain}

\numberwithin{equation}{subsection}

\def\boxit#1#2{\setbox1=\hbox{\kern#1{#2}\kern#1}%
\dimen1=\ht1 \advance\dimen1 by #1
\dimen2=\dp1 \advance\dimen2 by #1
\setbox1=\hbox{\vrule height\dimen1 depth\dimen2\box1\vrule}%
\setbox1=\vbox{\hrule\box1\hrule}%
\advance\dimen1 by .4pt \ht1=\dimen1
\advance\dimen2 by .4pt \dp1=\dimen2 \box1\relax}

\newcommand{\sur}[2]{\genfrac{}{}{0pt}{}{#1}{#2}}
\renewcommand{\theequation}{\thesubsection.\arabic{equation}}

\let\cal\mathcal

\def\AA{{\mathbf A}}

\def\CC{{\mathbf C}}

\def\LL{{\mathbf L}}

\def\NN{{\mathbf N}}

\def\PP{{\mathbf P}}
\def\QQ{{\mathbf Q}}
\def\RR{{\mathbf R}}

\def\ZZ{{\mathbf Z}}

\def\cC{{\mathcal C}}
\def\cD{{\mathcal D}}

\def\cL{{\mathcal L}}
\def\cM{{\mathcal M}}

\def\cV{{\mathcal V}}

\mathchardef\alphag="7C0B
\mathchardef\betag="7C0C
\mathchardef\gammag="7C0D
\mathchardef\deltag="7C0E
\mathchardef\varepsilong="7C22
\mathchardef\varphig="7C27
\mathchardef\psig="7C20
\mathchardef\zetag="7C10
\mathchardef\epsilong="7C0F
\mathchardef\rhog="7C1A
\mathchardef\taug="7C1C
\mathchardef\upsilong="7C1D
\mathchardef\iotag="7C13
\mathchardef\thetag="7C12
\mathchardef\pig="7C19
\mathchardef\sigmag="7C1B
\mathchardef\etag="7C11
\mathchardef\omegag="7C21
\mathchardef\kappag="7C14
\mathchardef\lambdag="7C15
\mathchardef\mug="7C16
\mathchardef\xig="7C18
\mathchardef\chig="7C1F
\mathchardef\nug="7C17
\mathchardef\varthetag="7C23
\mathchardef\varpig="7C24
\mathchardef\varrhog="7C25
\mathchardef\varsigmag="7C26
\mathchardef\Omegag="7C0A
\mathchardef\Thetag="7C02
\mathchardef\Sigmag="7C06
\mathchardef\Deltag="7C01
\mathchardef\Phig="7C08
\mathchardef\Gammag="7C00
\mathchardef\Psig="7C09
\mathchardef\Lambdag="7C03
\mathchardef\Xig="7C04
\mathchardef\Pig="7C05
\mathchardef\Upsilong="7C07

\DeclareMathOperator*{\lcm}{lcm}

\DeclareMathOperator{\sqa}{\square}
\def\sq{{\sqa\nolimits}}
\def\ord{{\rm ord}}
\def\Jac{{\rm Jac}}

\def\MV{{\rm MV}_\infty}

\begin{document}
\title[Metric properties
of definable
sets]{Local metric properties
and regular stratifications
of $p$-adic definable
sets}

\author{Raf Cluckers}
\address{Universit\'e Lille 1, Laboratoire Painlev\'e, CNRS - UMR 8524, Cit\'e Scientifique, 59655
Villeneuve d'Ascq C'edex, France, and,
Katholieke Universiteit Leuven, Department of Mathematics,
Celestijnenlaan 200B, B-3001 Leu\-ven, Bel\-gium\\  } \email{raf.cluckers@wis.kuleuven.be}
\urladdr{http://www.wis.kuleuven.be/algebra/Raf/}
\author{Georges Comte}

\address{Laboratoire J.-A. Dieudonn\'e,
Universit\'e de Nice - Sophia Antipolis, Parc Valrose,
06108 Nice Cedex 02, France (UMR 6621 du CNRS)}
\email{comte@math.unice.fr}
\urladdr{http://www-math.unice.fr/membres/comte.html}

\author{Fran\c cois Loeser}

\address{{\'E}cole Normale Sup{\'e}rieure,
D{\'e}partement de math{\'e}matiques et applications,
45 rue d'Ulm,
75230 Paris Cedex 05, France
(UMR 8553 du CNRS)}
\email{Francois.Loeser@ens.fr}
\urladdr{http://www.dma.ens.fr/$\sim$loeser/}

\maketitle

\renewcommand{\partname}{}

\begin{abstract} We study the geometry of germs of definable (semialgebraic or subanalytic) sets over a $p$-adic field from the metric, differential and measure geometric point of view.
We prove that the local density of such sets at each of their points does exist.
We then introduce the notion of distinguished tangent cone with respect to some open subgroup with finite index in the multiplicative group of our field and show, as it is the case in the real setting, that, up to some multiplicities, the local density may be computed on this distinguished tangent cone.
We also prove that these distinguished tangent cones stabilize for small enough subgroups.
We finally obtain the $p$-adic counterpart of the Cauchy-Crofton formula for the density. To prove these results we use the Lipschitz decomposition of definable $p$-adic sets of \cite{CCL} and prove here
the genericity of the regularity conditions  for stratification
such as $(w_f)$, $(w)$, $(a_f)$, $(b)$ and $(a)$ conditions.
\end{abstract}

\section*{Introduction}

The present paper is devoted
to the study of local metric properties of
definable subsets of the $p$-adic affine space, with special
stress on the local density of these subsets. It contains also results
on tangent
cones and existence of regular stratifications.

We shall start by recalling what is known in the real
and complex
context regarding the
local density of
(sub-)analytic
 sets.
 When $X_a$ is a germ at $a$ of
a complex analytic subset $X$
of real dimension $d$
of the affine space $\CC^n $, the
 local density $\Theta_{d}(X_a)$ of $X_a$,
sometimes called in this setting the Lelong number of $X_a$,
has been introduced by Lelong in \cite{Lelong} as the limit of volumes of
the intersection of a representative
$X$ of the germ with  suitably renormalized balls around $a$, namely
\begin{equation*}
\Theta_{d}(X_a)=\lim_{r\to 0}
\frac{\mu_{d}({X}\cap B(a,r))}{\mu_{d}(B^{d}(0,r))},
\end{equation*}
where $ B^{d}(0,r)$ is the real $d$-dimensional ball of centre $0$ and
radius $r>0$ and $\mu_d$ stands for $d$-dimensional volume.
Lelong actually proved that the function
$$\displaystyle r
\mapsto
\frac{\mu_d(X\cap B(a,r))}{r^d} $$
 decreases
as $r$ goes to $0$, pointing out, long before this concept has been formalized,
the tame behaviour of the local
normalized volume of analytic sets.
Ten years after Lelong's pioneering paper, Thie proved in
\cite{Thie} that the local density of a
complex analytic subset $X$ at a point $a$ is a positive integer
by expressing it as a sum
of local densities of the components, counted with multiplicities,
of the tangent cone
of $X$ at $a$. Finally, more than twenty years after Lelong's
definition,  Draper proved \cite{Draper} that the
local density is the algebraic multiplicity of the local ring of
$X$ at $a$.
The definition of local densities has been extended by
Kurdyka and Raby to real subanalytic subsets of $\RR^n$ in \cite{KR}
(see also \cite{KPR}). In fact,
although
 the arguments in \cite{KR} and \cite{KPR}
 which
prove
the existence of the density in the real subanalytic case
were
given before the notion of definable sets in
o-minimal structures
 emerged, they apply to the
real o-minimal setting. A short
proof of this existence result, again produced just before the concept
of tame definable sets
and involving the Cauchy-Crofton formula, may be found
in the seminal paper \cite{Lio} for semi-pfaffian sets.
In \cite{lk} (Theorem 1.3), one can find a proof
in the real o-minimal setting of the existence of the
local density viewed as the higher term of a finite
sequence of localized curvature invariants  involving the Cauchy-Crofton
formula and the theory of regular stratifications.
Of course, in the real  setting, the density
is in general no longer a positive integer, but a non
negative real number, and
Kurdyka and Raby proved an appropriate extension of  Thie's result by
 expressing again the local density in terms of the density of some
components of the real tangent cone. The existence of the local density
at each point of  the closure of a subanalytic set is a manifestation of
tameness of these
sets near their singular points. Similarly,
tameness in subanalytic geometry is also illustrated by the
tame behavior of
local density, viewed as a function of
the base-point of the germ at which it is computed:
it is actually proved in \cite{coliro} that this function  along a given
global subanalytic set is a Log-analytic function, that is, a polynomial
in subanalytic functions and their logarithms (see Siu's paper \cite{Siu}
for similar results in the complex case).
 Draper's result has been extended to the real setting by Comte in
 \cite{comte} in the following way. Recall that if $X$ is a complex
analytic subset   of the affine space
 of complex  dimension $d$ at $x$, the algebraic multiplicity of the
local ring
of $X$ at $x$
 is equal to the local degree of  a generic linear projection $p: X \rightarrow \CC^d$, that is, to the
  number of points near $x$ in a generic fiber of $p$.
  Over the reals, if $X$ is of local dimension $d$ at $x$,
  the number of points in fibers (near $x$) of a
  generic linear projection $p: X \rightarrow \RR^d$ is not generically constant near
  $p (x)$ in general.  The idea introduced in    \cite{comte}
to overcome this difficulty is to consider
  as a substitute for the local degree of $p$ the sum
  $d (p) :=\sum_{i \in \NN} i \cdot\theta_i$,
  where $\theta_i$ is the local density at $p(x)$ of the germ of the set of points in
  $\RR^d$ over which the  fiber of $p$ has exactly $i$ points near $x$.
  The so-called {\sl local Cauchy-Crofton formula} proved in
   \cite{comte} states that the average along all linear projections $p$
  of the degrees $d (p)$ is equal to the local density of $X$ at $x$ and
  can be considered as the real analogue of Draper's result.
  Finally, the complete multi-dimensional version of the local
Cauchy-Crofton formula for real subanalytic sets is presented in \cite{lk}
(Theorem 3.1),
where the multi-dimensional substitute of the $0$-dimensional
local degree $d(p)$ is obtained by considering the local
Euler characteristic of generic multi-dimensional fibers,
instead of the local number of points.

 Now let $K$ be a finite extension of $\QQ_p$ and
 $X$ be a definable subset
 (semi-algebraic or subanalytic)
  of
 $K^n$.
 Let $x$ be a point of $K^n$. When one tries to define the local density
 of $X$ at $x$ similarly to the archimedean case, one is faced to the
problem, illustrated in  \ref{fstart},  that the limit of local volumes
in general no longer exists. It appears that the normalized volumes
$v_n$ of $X$ in the balls
$B(x,n):=\{w\in K\mid \ord (x-w)\geq n\}$
 has a periodic
convergence, that is to say, there exists an integer $e>0$, such that
for all $c=0, \cdots, e-1$, $(v_{c+m\cdot e})_{m\in \NN} $ has a limit
$v_c$
in $\QQ $  (see Proposition \ref{lem:pos}), with possibly
$v_c\not=v_{c'}$, for $c\not=c'$.
 The reader having essentially in mind the real case
is thus strongly encouraged to start  reading this article by the
example studied in \ref{fstart} that emphasizes this phenomenon.
 We resolve that issue by using an appropriate
 renormalization device that leads us to express the mean value
$\displaystyle \frac{1}{e}\cdot \sum_{c=0}^{e-1}
v_c$ as the local density of $X$ at $x$.

Another new issue occurring in the $p$-adic
setting is the lack of a
natural
 notion of a tangent cone.
Unlike the real case where only the action of the multiplicative
group $\RR_+^\times$ has to be considered, in the $p$-adic case,
there seems to be
no preferential subgroup of $K^\times$ at hand.
We remedy this by introducing, for each definable open subgroup of
finite index
$\Lambda$  in $K^{\times}$, a tangent cone $C_x^{\Lambda}$ at $x$ which
is stable by homotheties
in $\Lambda$, that is,
which is
a $\Lambda$-cone.
 One should note that such
 $\Lambda$-cones were already considered more than twenty years ago  in
the work of Heifetz on
$p$-adic oscillatory integrals and wave front sets \cite{Heifetz}.
Nevertheless, we prove in
Theorem
\ref{distinguished cone} that, given a definable subset $X$,
among these cones, some are distinguished as
maximal for an inclusion property, and appear as the good tangent cones
to be considered, in the sense that they capture the local
geometry of our set. We are then able, by deformation to the tangent cone, to assign
multiplicities to points in the tangent cone $C_x^{\Lambda}$.

Our main results regarding $p$-adic local densities are
Theorem \ref{mt}, which is a $p$-adic analogue of the result
of Thie and Kurdyka-Raby, and Theorem \ref{lcc}, which is
a $p$-adic analogue of Comte's  local Cauchy-Crofton formula.
An important technical tool in our proof of
Theorem \ref{mt} is provided by Theorem
\ref{Lipschitz decomposition} which allows us to
decompose
our definable set into Lipschitz graphs. Such a regular decomposition
has been obtained in
 \cite{CCL} and extends to the $p$-adic setting
a real subanalytic result of \cite{KurWExp}.
In section \ref{swf} we prove the existence of $(w_f)$-regular stratifications for
definable $p$-adic functions, and consequently the existence of
Thom's $(a_f)$-regular stratifications for definable $p$-adic functions,
$(w)$-regular, or Verdier regular,  stratifications,
and Whitney's $(b)$-regular
stratifications for $p$-adic
definable sets (Theorem \ref{wf}).
\medskip

{During the preparation of this paper, the authors have been partially supported by grant
ANR-06-BLAN-0183. We also thank the Fields Institute of Toronto, where this
paper was partly written, for bringing us exceptional working conditions.}

\tableofcontents

\section{Preliminaries}

\subsection{Definable sets over the $p$-adics}\label{defpadic} Let $K$ be a finite field extension of $\QQ_p$ with valuation ring
$R$. We denote by $\ord$ the valuation and set $\vert x \vert :=
q^{- \ord (x)}$ and $\vert 0 \vert = 0$, with $q$ cardinality of
the residue field of $K$. If $x = (x_i)$ is a point in $K^m$ and
$n$ is an integer, we denote by $B(x,n)$ the ball
in $K^m$ given by the conditions
$\ord (z_i -
x_i) \geq n$, $1 \leq i \leq m$.

We recall the notion of (globally) subanalytic subsets of $K^n$
and of semi-algebraic subsets of $K^n$. Let $\cL_{\rm
Mac}=\{0,1,+,-,\cdot,\{P_n\}_{n>0}\}$ be the language of Macintyre and
$\cL_{\rm an}=\cL_{\rm Mac}\cup
\{^{-1},\cup_{m>0}K\{x_1,\ldots,x_m\}\}$, where $P_n$ stands for the
set of $n$th powers in $K^\times$,
where $^{-1}$ stands for the field inverse extended on $0$ by $0^{-1}=0$,
where $K\{x_1,\ldots,x_m\}$ is
the ring of restricted power series over $K$ (that is, formal power
series converging on
$R^m$),
 and each element $f$ of $K\{x_1,\ldots,x_m\}$ is interpreted
as the restricted analytic function $K^m\to K$ given by
\begin{equation}
x\mapsto
\begin{cases}
 f(x) & \mbox{if }x\in
R^m \\
0 & \mbox{else.}
\end{cases}\end{equation}
By subanalytic we mean $\cL_{\rm an}$-definable with coefficients
from $K$ and by semi-algebraic we mean $\cL_{\rm Mac}$-definable
with coefficients from $K$. Note that subanalytic,
resp.~semi-algebraic, sets can be given by a quantifier free formula
with coefficients from $K$ in the language $\cL_{\rm Mac}$,
resp.~$\cL_{\rm an}$.

In this section we let $\cL$ be either the language $\cL_{\rm
Mac}$ or $\cL_{\rm an}$ and by $\cL$-definable we will mean
semi-algebraic, resp.~subanalytic when $\cL$ is $\cL_{\rm Mac}$,
resp.~$\cL_{\rm an}$. Everything in this paper will hold for
both languages and we will give  appropriate references for
both languages when needed.

For each definable set $X\subset K^n$, let $\cC(X)$ be the
$\QQ$-algebra of functions on $X$ generated by functions $|f|$ and
$\ord(f)$ for all definable functions $f:X\to
K^\times$.

We refer to \cite{sd} and \cite{DvdD} for
the definition of the dimension of $\cL$-definable sets.

\subsection{The $p$-adic measure}\label{pmes}
Suppose that $X\subset K^n$ is an $\cL$-definable set of dimension
$d\geq 0$. The set $X$ contains a definable nonempty  open
$K$-analytic submanifold $X'\subset K^n$ such that $X\setminus X'$ has dimension
$<d$, cf.~\cite{DvdD}. There is a canonical $d$-dimensional
measure $\mu_d$ on $X'$ coming from the embedding in $K^n$, which is
constructed as follows, cf. \cite{serre}. For each $d$-element
subset $J$ of $\{1,\ldots,n\}$, with $j_i<j_{i+1}$, $j_i$ in $J$,
let $dx_J$ be the $d$-form $dx_{j_1}\wedge\ldots\wedge dx_{j_d}$ on
$K^n$, with $x=(x_1,\ldots,x_n)$ standard global coordinates on
$K^n$. Let $x_0$ be a point on $X'$ such that $x_I$ are local
coordinates around $x_0$ for some $I\subset\{1,\ldots,n\}$. For each
$d$-element subset $J$ of $\{1,\ldots,n\}$ let $g_J$ be the
$\cL$-definable function determined at a neigborhood of $x_0$ in
$X'$ by $g_Jdx_I=dx_J$. There is a unique volume form
$\vert \omega_{0}\vert_{X'}$ on $X'$ which is  locally equal to
$(\max_J|g_J|) \vert dx_I \vert$ around every point $x_0$ in $X'$.
Indeed, $\vert \omega_{0}\vert_{X'}$ is equal to $\sup_{J}\vert dx_J
\vert$. The canonical $d$-dimensional measure $\mu_d$ on $X'$ (cf.
\cite{serre} \cite{oesterle}), is the one induced by the volume form
$\vert \omega_{0}\vert_{X'}$. We extend this measure to $X$ by zero
on $X \setminus X'$
and still denote it by $\mu_d$.

\subsection{Adding sorts}\label{sr}
By analogy with the motivic framework, we now expand the language
$\cL$ to a three sorted language $\cL'$ having  $\cL$ as language
for the valued field sort, the ring language $\LL_{\rm Rings}$ for
the residue field, and the Presburger language $\LL_{\rm PR}$ for
the value group together with maps $\ord$ and $\ac$ as in \cite{cons}.
By taking the product of the measure $\mu_m$ with the
counting measure  on $k_K^n \times \ZZ^r$ one defines a measure
still denoted by $\mu_m$ on $K^m \times k_K^n \times \ZZ^r$.

One defines the dimension of an
$\cL'$-definable subset $X$ of $K^m\times k_K^n\times\ZZ^r$ as the
dimension of its projection  $ p(X) \subset K^m$. If $X$ is of dimension $d$,
one defines a measure $\mu_d$ on $X$ extending the previous construction  on $X$ by setting
\begin{equation}
\mu_d (W) := \int_{ p (X)} p_! (\11_W) \mu_d
\end{equation}
with $p_! (\11_W)$ the function $y \mapsto {\rm card}(p^{-1}(y)
\cap W)$.

For such an $X$, one defines $\cC (X)$ as the $\QQ$-algebra of
functions on $X$ generated by functions $\alpha$ and $q^{- \alpha}$
with $\alpha : X \rightarrow \ZZ$ definable in $\cL'$. Note that
this definition coincides with the previous one when $n = r = 0$.
Since $\cL'$ is interpretable in $\cL$, the formalism developed in
this section extends to $\cL'$-definable objects in a natural way.

\subsection{$p$-adic Cell Decomposition}
Cells are defined by induction on the number of variables
 \begin{definition}\label{def::cell}
An $\cL$-cell $A\subset K$ is a (nonempty) set of the form
\[
\{t\in K\mid |\alpha|\square_1 |t-c|\sq_2 |\beta|,\
  t-c\in \lambda P_n\},
\]
with constants $n>0$, $\lambda,c$ in $ K$, $\alpha,\beta$ in
$K^\times$, and $\square_i$ either $<$ or no condition. An $\cL$-cell
$A\subset K^{m+1}$, $m\geq0$, is a set  of the form
 \begin{equation}\label{Eq:cell:decay}
 \begin{array}{ll}
\{(x,t)\in K^{m+1}\mid
 &
 x\in D, \  |\alpha(x)|\sq_1 |t-c(
 x)|\sq_2 |\beta(x)|,\\
 &
  t-c(x)\in \lambda P_n\},
 \end{array}
 \end{equation}
 with $(x,t)=(x_1,\ldots,
x_m,t)$, $n>0$, $\lambda$ in $ K$, $D=p_m(A)$ a cell where $p_m$
is the projection $K^{m+1}\to K^m$, $\cL$-definable functions
$\alpha,\beta:K^m\to K^\times$ and $c:K^m\to K$, and $\square_i$
either $<$ or no condition, such that the functions
$\alpha,\beta$, and $c$ are analytic on $D$. We call $c$ the
center of the cell $A$ and $\lambda P_n$ the coset of $A$. In
either case, if $\lambda=0$ we call $A$ a $0$-cell and if
$\lambda\not=0$ we call $A$ a $1$-cell.
 \end{definition}

In the
$p$-adic semi-algebraic case, Cell Decomposition Theorems are  due to Cohen \cite{cohen}
and Denef \cite{D84}, \cite{Dcell} and they were extended in
\cite{Ccell}
to the subanalytic setting where one can find the
following version.
\begin{theorem}[$p$-adic Cell Decomposition]\label{thm:CellDecomp}
Let $X\subset K^{m+1}$ and $f_j:X\to K$ be $\cL$-definable for
$j=1,\ldots,r$. Then there exists a finite partition of $X$ into
$\cL$-cells $A_i$ with center $c_i$ and coset $\lambda_i P_{n_i}$
such that
 \begin{equation*}
 |f_j(x,t)|=
 |h_{ij}(x)|\cdot|(t-c_i(x))^{a_{ij}}\lambda_i^{-a_{ij}}|^\frac{1}{n_i},\quad
 \mbox{ for each }(x,t)\in A_i,
 \end{equation*}
with $(x,t)=(x_1,\ldots, x_m,t)$, integers $a_{ij}$, and
$h_{ij}:K^m\to K$ $\cL$-definable functions which are analytic on
$p_m(A_i)$, $j=1,\ldots,r$. If $\lambda_i=0$, we use the
convention that $a_{ij}=0$.
 \end{theorem}

Let us also recall the following lemma from \cite{Cexp}.
\begin{lem}\label{prop:descrip:simple}
Let $X\subset K^{m+1}$ be $\cL$-definable and let $G_j$ be
functions in $\cC(X)$
in
the variables $(x_1,\ldots,x_m,t)$ for
$j=1,\ldots,r$. Then there exists a finite partition of $X$ into
$\cL$-cells $A_i$ with center $c_i$ and coset $\lambda_i P_{n_i}$
such that each restriction $G_j|_{A_i}$ is a finite sum of
functions of the form
\[
|(t-c_i(x))^{a}\lambda_i^{-a}|^\frac{1}{n_i}\ord (t-c_i(x))^{s}h(x),
\]
where $h$ is in $\cC(K^m)$, and $s\geq 0$ and $a$ are integers.
\end{lem}

The following $p$-adic curve selection lemma is due to van den Dries
and Scowcroft \cite{sd} in the semi-algebraic case and to Denef and
van den Dries \cite{DvdD} in the subanalytic case. The statement is
the $p$-adic counterpart of the semi-algebraic or subanalytic curve selection
lemma over the reals.

\begin{lem}[Curve Selection]\label{curve}
Let $A$ be a definable subset of $K^n$ and let $x$ be in $\overline
A$. Then there exists a definable function $f=(f_1,\ldots,f_n):R\to
K^n$ such that the $f_i$ are given by power series (over $K$)
converging on $R$, such that $f(0)=x$, and such that $f(R\setminus
\{0\})\subset A$.
\end{lem}

The following is a $p$-adic analogue of a classical lemma by Whitney (see \cite{Whitney}).
\begin{lem}[$p$-adic  Whitney Lemma]\label{lemWhitney}Let $g:R\to K^n$ be a map given by $n$ analytic power series over
$K$, converging on $R$, such that the map $g$ is nonconstant.
Then, the limit $\ell\in\PP^{n-1}(K)$ for $r\to 0$ of the lines
$\ell_r\in\PP^{n-1}(K)$ connecting $g(0)$ with $g(r)$ exists. Also
the limit $\ell'$ of the tangent lines $\ell_r':=\{g(r)+\lambda
(\partial g_1/\partial r,\ldots,\partial g_1/\partial r)_{|r}\mid
\lambda\in K\}$ for $r \to 0 $ exists and $\ell'=\ell$.
\end{lem}
\begin{proof}
Since $g$ is nonconstant, for $r\not=0$ close to $0$ one has
$g(r)\not = g(0)$ and $(\partial g_1/\partial r,\ldots,\partial
g_1/\partial r)_{|r}\not =0$, and hence, $\ell_r$ and $\ell'_r$ are
well-defined for $r\not=0$ close to $0$. We may suppose that
$g(0)=0$ and that each of the $g_i$ is nonconstant. Write
$g_i(r)=\sum_{j\geq 0} a_{ij}r^j$ with $a_{ij}\in K$ and for each
$i$, let $k_i$ be the smallest index $j$ such that $a_{ij}\not=0$.
Then $k_i>0$ for each $i$ since $g(0)=0$. Let $k$ be the minimum of
the $k_i$.
 Then clearly $\ell$ and $\ell'$ are the same line $\ell$
connecting $0$ and $(a_{1k},\ldots,a_{nk})\not=0$. Indeed, the line
$\ell_r$ connects $0$ and $g(r) $ which is equivalent to connecting
$0$ and $g(r)/r^k$; the point $g(r)/r^k$ converges to
$(a_{1k},\ldots,a_{nk})$ and thus $\ell_r$ converges to the line
$\ell$. Likewise, the line $\ell'_r$ connects $0$ and $(\partial
g_1/\partial r,\ldots,\partial g_1/\partial r)_{|r}$ which is
equivalent to connecting $0$ and
 $$
 \frac{1}{kr^k} (\partial g_1/\partial r,\ldots,\partial
 g_1/\partial r)_{|r};
$$
 the point $\frac{1}{kr^k} (\partial
g_1/\partial r,\ldots,\partial g_1/\partial r)_{|r}$ converges to
$(a_{1k},\ldots,a_{nk})$ and thus also $\ell_r'$ converges to $\ell$
when $r\to 0$.
\end{proof}

\subsection{}Fix two integers $d \leq m$. Let $U$ be an open definable subset of
$K^d$ and let $\varphi$ be a definable analytic mapping $U
\rightarrow K^{m - d}$. We view the graph  $\Gamma (\varphi)$ of
$\varphi$ as a definable subset of $K^m$. Let $\varepsilon$ be a
positive real number. We say that $\varphi$
is $\varepsilon$-analytic,
if the norm $| D \varphi | = \max_{i,j}
|\partial \varphi_i/\partial x_j|$ of the differential of $\varphi$
is less or equal than $\varepsilon$ at every point of $U$.

For $\varepsilon>0$, call a function $f:D\to K^{m}$ on a subset $D$
of $K^n$ $\varepsilon$-Lipschitz when for all $x, y \in D$ one has
$$
|f(x)-f(y)|\leq \varepsilon |x-y|.
$$
The function $f$ is called locally $\varepsilon$-Lipschitz when for
each $x\in D$ there exists an open subset $U$ of $K^n$ containing $x$ such
that the restriction of $f$ to $U\cap D$ is $\varepsilon$-Lipschitz.

\begin{lem}\label{ana-lip}
Let $U$ be open in $K^n$ and $f:U\to K^m$ a function which is
$\varepsilon$-analytic. Then $f$ is locally $\varepsilon$-Lipschitz.
\end{lem}
\begin{proof}
Choose $u\in U$, and a basic neighborhood $U_u$ of $u$ in $U$ such
that the component functions $f_i$ of $f$ are given by converging
power series on $U_u$, where basic neighborhood means a ball of the
form
$c+ \lambda R^n$ with $c\in K^n$ and $\lambda \in K^\times$.
We may suppose
that $U_u=R^n$, that $u=0$, and that $\varepsilon=1$. We may also assume
that for each $i,j$, the partial derivative $\partial
f_j(x)/\partial x_i $ is bounded in norm by $1$ on $U_u$. Since
$|\partial f_j(x)/\partial x_i (0)|\leq 1$, it follows that the
linear term of $f_j$ in $x_i$ has a coefficient of norm $\leq 1$ for
each $i,j$. By the convergence of the power series, the coefficients
of the $f_j$ are bounded in norm, say by $N$, and we can put
$U':=\{x\in R^n\mid |x|<1/N\}$. Clearly $U'$ contains $u=0$. By the
non-archimedean property of the $p$-adic valuation, the restriction
of $f$ to $U'$ is $1$-Lipschitz.
\end{proof}
For more results related to Lipschitz continuity on the $p$-adics, see \cite{CCL} or Theorem \ref{Lipschitz decomposition} below.
The following lemma is a partial converse of Lemma \ref{ana-lip},
especially in view of the fact that any definable function is
piecewise analytic.

\begin{lem}\label{ana-eps}
Let $U$ be a definable open in $K^n$ and let $f:D\to K^m$ be a definable
analytic function which is locally $\varepsilon$-Lipschitz. Then $f$
is $\varepsilon$-analytic.
\end{lem}
\begin{proof}
We proceed by contradiction. Suppose that $|D f|>\varepsilon$ at
$u\in U$. Choose a basic neighborhood $U_u$ of $u$ in $U$ such that
the component functions $f_i$ of $f$ are given by converging power
series on $U_u$, (here again by basic neighborhood we mean a ball of the form
$c+\lambda R^n$ with $c\in K^n$ and $\lambda \in K^\times$). We may suppose that
$U_u=R^n$, that $u=0$, and that $\varepsilon=1$. By assumption, we have for some $i,j$ that $|(\partial
f_j(x)/\partial x_i)(0)|>1 $, hence, the linear term of $f_j$ in
$x_i$ has a coefficient of norm strictly greater than $1$. By the
convergence of the power series, the coefficients of the $f_j$ are
bounded in norm, say by $N$, and thus for any $x$ in $\{x\in R^n\mid
|x|<1/N\}$ one has $|f(0)-f(x)|>|x|$ which contradicts the fact that
$f$ is $\varepsilon$-Lipschitz with $\varepsilon=1$.
\end{proof}

The following is the $p$-adic analogue of Proposition 1.4 of
\cite{KR}.
\begin{prop}\label{1.4} Let $X$ be a definable subset of dimension
$d$ of $K^m$. For every $\varepsilon > 0$, there exists a definable
subset $Y$ of $X$ of dimension $< d$, $N (\varepsilon) \geq 0$,
definable open subsets $U_i (\varepsilon)$ of $K^d$, for $1 \leq i
\leq N (\varepsilon)$,
definable, $\varepsilon$-analytic functions
$\varphi_i (\varepsilon) : U_i (\varepsilon) \rightarrow K^{m - d}$, with graphs
$\Gamma_i (\varepsilon)$,
and elements $\gamma_1$, \dots, $\gamma_{N (\varepsilon)}$ in
$\GL_m (R)$ such that the sets $\gamma_i ( \Gamma_i (\varepsilon))$ are all disjoint and contained in $X$, and
$$X =  \bigcup_{1 \geq i \geq N (\varepsilon)} \gamma_i (\Gamma_i (\varepsilon))
\cup Y.$$
\end{prop}

\begin{proof}
Take a finite partition of $X$ into cells $X_i$. By neglecting
cells of dimension $<d$, we may suppose that $X=X_1$ and that
$X_1$ is a cell of type $(i_1,\ldots, i_m)$ with $\sum_{j=1}^m
i_j=d$.
 By reordering the coordinates, by partitioning again into
cells, and neglecting cells of dimension $<d$, we may suppose that
$X$ is a cell of type
 $$(i_1,\ldots,
i_m)=(1,\ldots,1,0,\ldots,0).
 $$
 Let $p:K^m\to K^d$ be the
projection to the first $d$ coordinates.

The set $X$ itself is already a graph of a map
$(f_1,\ldots,f_{m-d})=f:p(X)\to K^{m-d}$. We may suppose that the
$f_i$ are analytic. Write $(x_1,\ldots,x_d)$ for the coordinates on
$p(X)$. Consider the differentials $(\frac{\partial f_j}{\partial
x_i})_{i=1}^d$.
 By partitioning further, we may suppose that for each $i=1,\ldots, d$,
 either
$|\frac{\partial f_j}{\partial x_i}| \leq \varepsilon$
or
$|\frac{\partial f_j}{\partial x_i}| > \varepsilon$
hold on $p(X)$.
 If $|\frac{\partial f_j}{\partial x_i}|> \varepsilon$ on $p(X)$ for some
$i$ and $j$, then we may suppose that
$$
 (x_1,\ldots,x_d)\mapsto (x_1,\ldots,x_{i-1},\frac{f_j(x_1,\ldots,x_d)}{\varepsilon},x_{i+1},\ldots,x_{d})
 $$
is a bi-analytic change of variables
(cf.~Corollary 3.7 of \cite{CCL}).
By performing such change of
variables successively, we are done by the chain rule for differentiation.
\end{proof}

The following lemma is classical, see, for example, \cite{DHM} for
the semi-algebraic case and, for example, \cite{CLR1} for the
subanalytic case. Let $\cL^\ast$ be the language $\cL$ together with a
function symbol for the field inverse on $K^\times$ (extended by
zero on zero), function symbols for each $n$ which stands for a
(definable) $n$-th root picking function $\sqrt[n]{}$ on the $n$-th
powers (extended by zero outside the $n$-th powers), and for each
degree $n$ a Henselian root picking function $h_n$ for polynomials
of degree $n$ in the $n+1$ coefficients (extended by zero if the
conditions of Hensel's Lemma are not fulfilled).
\begin{lem}[\cite{DHM}, \cite{CLR1}]\label{terms}
Let $f:D\subset K^n\to K^m$ be an $\cL$-definable function. Then $D$
can be partitioned into finitely many definable pieces $D_i$ such
that there are $\cL^\ast$-terms $t_i$ with $f(x)=t_i(x)$ for each $i$
and each $x\in D_i$.
\end{lem}

\section{Local densities}

\subsection{A false start}\label{fstart} Let $X$ be a definable subset
of $K^m$ of dimension $d$ and let
$x$ be a point of $K^m$. Considering what is already known in the
complex analytic and real o-minimal case,
a natural way to define
the local density of $X$ at $x = (x_1, \dots, x_m)$ would be
to consider the limit
of $q^{nd} \mu_d (X \cap B (x, n))$, as $n \to \infty$.
Unfortunately this na\"\i ve attempt fails as
is
shown by the following
example that
we  present
in detail in order to caution the reader not to rely too heavily on
intuition coming from the real setting.
Take  $X$ the subset of points of even valuation in $K$ and $x = 0$.
Write $\pi_K$ for a uniformizer of $R$.
The unit ball $B$ in $K$ being of measure $1$, the ball
$\pi_K^\ell\cdot B = B(0,\ell)$
of radius $q^{-\ell}$ has volume $q^{-\ell}$ and,
by consequence, the
sphere
$\pi_K^\ell\cdot S=\pi_K^\ell\cdot B\setminus \pi_K^{\ell+1}\cdot B$
of radius $q^{-\ell}$ has volume
$q^{-\ell}(1-q^{-1}) $. For $k\in \NN$, let us first
compute the volume of
$X\cap B(0,2k)$. The set
 $X\cap B(0,2k)$ is the disjoint union of the spheres
 $\pi_K^{2j}\cdot S$ for $j\ge k$ and thus has as volume
 $$ \mu_1(X\cap B(0,2k))=(1-q^{-1})(q^{-2k}+q^{-2k-2}+\cdots)
=\frac{q^{-2k}}{1+q^{-1}}.$$
On the other hand, the set
 $X\cap B(0,2k-1)$ is also the disjoint union of the spheres
 $q^{2j}\cdot S$ for $j\ge k$ and thus has as volume
 $$ \mu_1(X\cap B(0,2k-1))=\frac{q^{-2k}}{1+q^{-1}}.$$
We finally see that in this example
the value of the limit
$\displaystyle \lim_{\ell\to \infty}
 \frac{\mu_1(X\cap B(0,\ell))}{\mu_1(B(0,\ell))} $
depends on the parity of $\ell$, since
 $$\lim_{k \to \infty} q^{2k} \mu_1 (X \cap B (0, 2k)) =
(1 + q^{- 1})^{-1}$$
and
$$\lim_{k \to \infty} q^{2k - 1} \mu_1 (X \cap B (0, 2k - 1)) =
(1 + q)^{-1}.$$
In our example one notices that the convergence of
the ratio $\displaystyle
 \frac{\mu_1(X\cap B(0,\ell))}{\mu_1(B(0,\ell))} $  is
 $2$-periodic and that  one may recover the expected local density,
which should be $\displaystyle \frac{1}{2}$,
by taking the average of the two limits.
To obviate the kind of difficulty presented by this example
(the periodic convergence), we
are led to introduce a regularization device that we shall explain now.

\subsection{Mean value at infinity of bounded constructible functions}

We will use the following elementary definition of the mean value at infinity of certain real valued functions on $\NN$.
\begin{definition}\label{def:MV}
Say that a function $h:\NN\to \RR$ has a mean value at infinity if there exists an integer $e>0$ such
that
$$
\lim_{\sur{n\to\infty}{n\equiv c\bmod e}} h(n)
$$
exists in $\RR$ for each $c=0,\ldots,e-1$ and in this case define the mean value at infinity of $h$ as the average
$$
\MV (h) := \frac{1}{e} \sum_{c=0}^{e-1}\lim_{\sur{n\to\infty}{n\equiv c\bmod e}} h(n).
$$
Clearly the value $\MV (h)$ is independent of the choice of the modulus $e>0$.
\end{definition}

Let $X$ be a definable subset of $K^m$, so that
$X \times \NN$
is a definable subset of
$K^m \times \ZZ$.
Say that a real valued function $g$ on $X \times
\NN$ is $X$-bounded if for every $x$ in $X$ the restriction of
$g$ to $\{x\} \times \NN$ is bounded (in the sense that $g(\{x\} \times \NN)$ is contained in a compact subset of $\RR$).
 As has been indicated in the introduction and in the example of section \ref{fstart}, for an $X$-bounded function $\varphi$ in $\cC (X \times
\NN)$ and $x\in X$, the function $\varphi_x:\NN\to\QQ:n\mapsto \varphi(x,n)$ may not have a unique limit for $n\to\infty$, but it may have a mean value at infinity $\MV(\varphi_x)$, as we will indeed show in Proposition \ref{lem:pos}. We will moreover show in Proposition \ref{lem:pos} that $\MV(\varphi_x)$, considered as a function in $x\in X$, lies in $\cC(X)$ and that $\MV(\varphi_x)$ can be calculated using a single integer $e$ as modulus when  $x$ varies in $X$.

\begin{lem}\label{lem:MV}
Let $\varphi$ be in in $\cC (X \times \NN)$. Suppose that, for each $x\in X$, the function $\varphi_x:\NN\to\QQ:n\mapsto \varphi(x,n)$ has finite image. Then  $\varphi_x$ has a mean value at infinity $\MV (\varphi_x)$ for each $x$. Moreover, there exist a definable function
$b:X\to \NN$ and an integer $e>0$ such that for all $c$ with
$0\leq c< e$ and all $x\in X$, the rational number $d_c(x):=\varphi(x,n)$ is independent of $n$ as long as
$n\geq b(x)$ and $n\equiv c \bmod e$. Thus, for each $x\in X$, one has
$$
\MV (\varphi_x) = \frac{1}{e} \sum_{c=0}^{e-1}d_c(x).
$$
By consequence, the function $\MV (\varphi_x)$, considered as a function in $x\in X$, 
lies in $\cC(X)$.
\end{lem}
\begin{proof}
The lemma is a direct consequence of Lemma \ref{prop:descrip:simple} and quantifier elimination in the three sorted language $\cL'$ of section \ref{sr}. Indeed, for $\varphi\in\cC(X\times \NN)$ there exist, by Lemma \ref{prop:descrip:simple} and quantifier elimination in $\cL'$, a definable function
$b:X\to \NN$ and an integer $e>0$ such that for all $c$ with
$0\leq c< e$ one has
\begin{equation}\label{varphixn}
\varphi(x,n) = \sum_{i=1}^k n^{\ell_i} q^{a_{i}n} h_{ic}(x)
\end{equation}
for all $x\in X$ and all $n$ with $n\geq b(x)$ and $n\equiv c \bmod e$, and where the $h_{ic}$ are in $\cC(X)$. Clearly, by regrouping, we may suppose that the pairs $(\ell_i,a_i)$ are mutually different. But then, since $\varphi_x$ has finite image for each $x\in X$, one must find
$$
\varphi(x,n) = h_{jc}(x)
$$
for all $x\in X$ and all $n$ with $n\geq b(x)$ and $n\equiv c \bmod e$, where $j$ is such that  $(\ell_j,a_j)=(0,0)$. Hence, one has $h_{jc}=d_c$ and we are done.
\end{proof}

\begin{prop}\label{lem:pos}
Let $\varphi$ be in $\cC (X \times \NN)$. Suppose that $\varphi$ is $X$-bounded.
Then there exist $\varphi'$ in $\cC (X \times \NN)$ with $\lim
_{n\to\infty}\varphi'(x,n)=0$ for all $x\in X$ and such that the function
$$
g_x:  \NN\to\QQ:n\mapsto \varphi(x,n) -  \varphi'(x,n)
$$
has finite image. Clearly, the function $g:X\times \NN:(x,n)\mapsto g_x(n)$ lies in $\cC(X\times \NN)$.
Hence, $\MV(g_x)$ and $\MV(\varphi_x)$ exist and are equal and the function $\MV (\varphi_x)$, considered as a function in $x\in X$,
lies in $\cC(X)$.
Also, if $\varphi\geq 0$ then $\MV (\varphi_x) \geq
0$ for all $x\in X$.
\end{prop}
\begin{proof}
Write again $\varphi$ as in (\ref{varphixn}) for some integer $e$, where again the pairs $(\ell_i,a_i)$ are mutually different. Define $\varphi'(x,n)$ as the partial sum
$$
\sum_{i\in I} n^{\ell_i} q^{a_{i}n} h_{ic}(x)
$$
for $x\in X$ and $n$ satisfying $n\geq b(x)$ and $n\equiv c \bmod e$, where $c=0,\ldots,e-1$, where $I$ consists of those $i$ with $a_i<0$. Extend $\varphi'$ to the whole of $X\times \NN$ by putting it equal to $\varphi$ for those $n$ with $n<b(x)$. Since $\varphi$ is $X$-bounded, one must have that $a_i\leq 0$ for all $i$, and, for those $i$ with $a_i=0$ one must have $\ell_i=0$. But then, we find
$$
g(x,n) = h_{jc}(x)
$$
for all $x\in X$ and all $n$ with $n\geq b(x)$ and $n\equiv c \bmod e$, where $j$ is such that  $(\ell_j,a_j)=(0,0)$. For $n$ with $n<b(x)$ one clearly has $g(x,n)=0$. The conclusions now follow from Lemma \ref{lem:pos}.
\end{proof}

\subsection{Local densities}
As already sketched in the introduction, we will define the local density of an $\cL$-definable set $X\subset K^m$ at a point $x$ as the mean value at infinity of the renormalized measure of the intersection of $X$ with the sphere of radius $q^{-n}$ around $x$. At our disposal to show that this is well-defined we have  Proposition \ref{lem:pos} and Lemma \ref{2.3.1} below which guarantee the existence of the mean value at infinity. More generally, for a bounded function $\varphi$ in $\cC(X)$, we extend $\varphi$ to $K^m$ by zero outside $X$ and we will define the density of $\varphi$ at any point $x\in K^m$ by a similar procedure, replacing the measure by an integral of $\varphi$ on a small sphere around $x$.

\subsection*{}
Let $\varphi$ be a bounded function in $\cC (X)$, meaning that
the image of $\varphi$ is contained in a compact subset of $\RR$.
For $(x, n ) $ in
$K^m \times \NN$ we set
\begin{equation}\label{gn}
\gamma (\varphi) (x, n) :=
\int_{S(x,n)\cap X} \varphi (y) \mu_d,
\end{equation}
where $S(x,n)$ is the sphere $\{y\in K^m\mid |x-y|=q^{-n}\}$ of radius $q^{-n}$ around $x$. Note that, by Lemma \ref{sph=ball} below, one could as well work with balls around $x$ instead of spheres consequently in this section.
 By \cite{D85} for the semi-algebraic case and \cite{Ccell} for the
subanalytic case, the function $\gamma (\varphi) : (x, n) \mapsto
\gamma (\varphi) (x, n) $ lies in $\cC (K^m \times \NN)$.

Suppose that $X$ is of dimension $d$.
Then we renormalize $\gamma (\varphi)$ by dividing it by the volume of the $d$-dimensional sphere of corresponding radius and define the resulting function
$\theta_d (\varphi)$ by
\begin{equation}\label{dp}
\theta_d (\varphi) (x, n) :=
 \frac{\gamma (\varphi) (x, n)}{\mu_d (S_{d}(n))},
\end{equation}
where $S_{d}(n)$ is the $d$-dimensional sphere of radius $q^{-n}$, namely the set $\{w\in K^d\mid |w| = q^{-n}\}$. Note that $S_{d}(n)$ has measure equal to  $(1-q^{-d}) q^{-nd}$ and thus, $\theta_d (\varphi)$ lies in $\cC (K^m \times \NN)$.

The following lemma yields sufficient conditions for the mean value at infinity of $\theta_d (\varphi)$ to exist, in view of Proposition \ref{lem:pos}.
\begin{lem}\label{2.3.1} Let $\varphi$ be a bounded function in $\cC (X)$. Assume $X$ is of dimension $d$. Then
the function
$\theta_d (\varphi)$ lies in $\cC (K^m \times \NN)$ and
is $K^m$-bounded.
\end{lem}
\begin{proof} That $\theta_d (\varphi)$ lies in $\cC (K^m \times \NN)$ is shown above, so we just have to show that $\theta_d (\varphi)$ is $K^m$-bounded. By the additivity of integrals and by cell decomposition, we may suppose that $X$ is a cell of dimension $d$. By changing the order of the coordinates if necessary and by Proposition \ref{jacprop}, we may suppose that $X$ projects isometrically to the first $d$ coordinates of $K^m$.
 If now $M>0$ is such that $\varphi(y)$ lies in the real interval $[-M,M]$ for
all $y\in X$, then it is clear by construction that $\theta_d (\varphi)(x)$ also lies in $[-M,M]$ for
all $x\in K^m$.
\end{proof}

It follows from  Lemma \ref{2.3.1} and Proposition \ref{lem:pos} that if $\varphi$ is a bounded function
in $\cC (X)$ one can set
\begin{equation}\label{Dd}
\Theta_d (\varphi) := \MV \theta_d (\varphi),
\end{equation}
that is, for $x\in K^m$,  $\Theta_d (\varphi)(x)$ is the mean value at infinity of the function $n\mapsto \theta_d (\varphi)(x,n)$.
By Proposition \ref{lem:pos}, the function $\Theta_d (\varphi)$ lies in $\cC (K^m)$.
For $x$ in $K^m$, we call
$\Theta_d (\varphi) (x)$
the local density of $\varphi$ at $x$.
More generally, if $\varphi$ is bounded on a neighborhood of some $x\in K^m$, then
$\Theta_d (\varphi) (x)$ can be defined by first extending $\varphi$ by zero outside of this neighborhood and calculate its local density by the above definitions which is clearly independent of the choice of the neighborhood.
One should also note that $\Theta_d (\varphi) (x)$ is zero when $x$ does not belong to the closure of $X$.

\begin{definition}\label{pTheta} Let  $X$ be a definable subset of $K^m$ of dimension $d$ and let $x$ be a point in $K^m$.
We call the rational number
$$\Theta_d (X) (x) :=
\Theta_d (\11_X) (x)$$
the local density of $X$ at $x$,
where $\11_X$ is the characteristic function of $X$ which clearly lies in $\cC(K^m)$.
\end{definition}

Note that Definition \ref{pTheta} resembles the definition of the complex and real density as given in the introduction, where instead of the limit $\lim_{r\to 0}$ one takes $\MV$.

\begin{lem}\label{sph=ball}
Renormalizing with balls instead of with spheres yields the same local density functions $\Theta_d$. Precisely, for $\cL$-definable $X$ of dimension $d$ and
for $\varphi$ a bounded function
in $\cC (X)$ one has for $x\in K^m$
 $$
 \Theta_d (\varphi) = \MV (  \theta'_d (\varphi)  ),
 $$
 where
 $$
 \theta'_d (\varphi) (x, n) :=
 \frac{\gamma' (\varphi) (x, n)}{\mu_d (B_{d}(n))},
 $$
 $$
 \gamma' (\varphi) (x, n) :=
\int_{B(x,n)\cap X} \varphi (y) \mu_d,
 $$
and where $B_{d}(n)$ is the $d$-dimensional ball of radius $q^{-n}$, namely $\{w\in K^d\mid |w|\leq q^{-n}\}$, and $B(x,n)$ is the ball $\{y\in K^m\mid |x-y|\leq q^{-n}\}$ around $x$ as defined in \ref{defpadic}. In particular, $\theta'_d (\varphi)$ lies in $\cC(K^m\times \NN)$ and is $K^m$-bounded and thus its $\MV$ is well-defined.
\end{lem}
\begin{proof}
That $\theta'_d (\varphi)$ lies in $\cC(K^m\times \NN)$ and is $K^m$-bounded is proven as Lemma \ref{2.3.1}. We have to prove that $ \Theta_d (\varphi) = \MV (  \theta'_d (\varphi)  )$, that is, for $x\in K^m$,  $\Theta_d (\varphi)(x)$ is the mean value at infinity of the function $n\mapsto \theta'_d (\varphi)(x,n)$.
It is clear that
$$
  \gamma (\varphi) (x, n) =  \gamma' (\varphi) (x, n) -  \gamma' (\varphi) (x, n+1)
$$
and that
$$
\theta_d (\varphi) (x, n) =   \frac{1}{(1-q^{-d})}  \big( \theta_d' (\varphi) (x, n) -  q^{-d}\theta_d' (\varphi) (x, n+1) \big) .
$$
Now we are done by the following fact, which holds for any real constant $b\not=1$. If a function $f:\NN\to\RR$ has a mean value at infinity, then so does $g:\NN\to\RR:n\mapsto \frac{1}{1-b} ( f(n) - b f(n+1) )$, and their mean values at infinity are equal.
\end{proof}

\begin{example}
Let us note that in the example of \ref{fstart} of points of even
valuation in $K$, one gets $\Theta_1 (X) (0) = \frac{1}{2}$. More
generally, if $\Lambda$ is a  definable open subgroup of finite
index $r$ in $K^{\times}$ and $y$ is a point in $K^{\times}$, we
have $\Theta_1 (\Lambda y) (0) = \frac{1}{r}$. Indeed, it is
easily checked that $\Theta_1 (\Lambda y) (0)$ does not depend on
$y$, hence if $y_1$, \dots, $y_r$ is a set of representative of
$K^{\times} / \Lambda$, we have $1 = \Theta_1 (\cup_{1 \leq i \leq
r} \Lambda y_i) (0) = r \Theta_1 (\Lambda) (0)$.
\end{example}

\begin{prop}Let $\varphi$ be a bounded function in $\cC (X)$ and assume $X$ is of dimension $d$.
Denote by $\tilde \varphi$ the extension of $\varphi$ by zero on $K^m$.
Then the support of $\tilde \varphi - \Theta_d (\varphi) $ is contained in an $\cL$-definable set of dimension $< d$.
\end{prop}
\begin{proof}
Suppose $X\subset K^m$. Since for $x$ not in  $\overline X$, for
all sufficiently large $n$ one has that $\theta_d (\varphi) (x,
n)=0$,
the support of
$\Theta_d (\varphi)$ is contained in the closure $\overline X$ of $X$
in $K^n$. After removing a subset of dimension $<d$ we may assume $X$ is a smooth subvariety and $\varphi$
is locally constant (for example after an iterated application of Lemma \ref{prop:descrip:simple}), in which case the result is clear.
\end{proof}

\subsection{}For further use we shall give some basic properties of local densities.

\begin{prop}\label{bord}
Let $X$ be definable subset  of dimension $d$ of $K^m$. Then
$$\Theta_d (X) (x) =
\Theta_d (\overline X) (x),
$$
where $\overline X$ denotes the closure of $X$.
\end{prop}

\begin{proof}Indeed, by additivity it is enough to prove
that
$\Theta_d (\overline X \setminus X) (x) = 0$, which follows from the fact that
$\Theta_d (Y) (x) = 0$ when $Y$ is definable of dimension $< d$.
\end{proof}

%
%

\begin{prop}\label{convmon} Let $X$ be an $\cL$-definable set of dimension $d$ and let $M>0$ be a constant. Consider a series of functions $\varphi_n:X\to\RR$, $n \in \NN$, such that the function $(x, n) \mapsto \varphi_n (x)$ lies in $\cC (X \times \NN)$
and such that
$0 \leq \varphi_n\leq \varphi_{n+1} \leq \dots \leq M$ for all $n$.
Then the function $\varphi$ defined as $\sup \varphi_n$ lies in $\cC(X)$ and is bounded.
Moreover,
$$\Theta_d (\varphi) (x) = \lim_n \Theta_d (\varphi_n) (x)$$
and
$$
0\leq \Theta_d (\varphi_n) (x)\leq \Theta_d (\varphi_{n+1}) (x)
$$
for each $n$ and $x$.
\end{prop}

\begin{proof}
Clearly the function $\varphi$ is bounded and lies in $\cC(X)$ by Lemma \ref{lem:pos}.
Note that
$$
 \gamma (\varphi)(x,m)  = \lim_n \gamma (\varphi_n)(x,m)
$$
for each $m$ and $x$ by the Monotone Convergence Theorem.
Hence, by the definition of $\theta_d$, also
\begin{equation}\label{thetan}
 \theta_d (\varphi)(x,m)  = \lim_n \theta_d (\varphi_n)(x,m)
 \end{equation}
for each $m$ and $x$.
Clearly
\begin{equation}\label{nleq}
0\leq \gamma (\varphi_n)\leq  \gamma (\varphi_{n+1})\
\mbox{ and }\
0\leq \theta_d (\varphi_n)\leq \theta_d (\varphi_{n+1})
\end{equation}
for all $n$, on the whole of $X$, and hence
$$
0\leq \Theta_d (\varphi_n) (x)\leq \Theta_d (\varphi_{n+1}) (x),
$$
by the definition of $\Theta_d$.
Now the equality $\Theta_d (\varphi) (x) = \lim_n \Theta_d (\varphi_n) (x)$ follows from (\ref{thetan}), (\ref{nleq}), and the definitions of $\MV$ and $\Theta_d$, by changing the order of limits over $n$ and over $m$.
\end{proof}

\section{Tangent cones}
\subsection{Cones}
We shall consider the set
$\cD$ of open subgroups of $K^\times$ which are of finite index in
$K^{\times}$. We order $\cD$ by inclusion. Note that for each $n>0$, the group $P_n$ of the $n$th powers in $K^\times$ lies in $\cD$, and any $\Lambda$ in $\cD$ equals, as a set, a finite disjoint union of cosets of some $P_n$, see Lemma \ref{cones1}, and is thus $\cL$-definable.
 We shall say a certain property (P) holds for $\Lambda$ small enough, if there exists $\Lambda_0$ in $\cD$ such that (P) holds for every $\Lambda\in \cD$ contained in $\Lambda_0$.

Let $\Lambda$ be a subgroup of $K^\times$ in $\cD$.
It acts naturally on $K^n$ by multiplicative translation $\lambda \cdot  z := \lambda z$, that is, by scalar multiplication on the vector space $K^n$.
By a $\Lambda$-cone in $K^n$ we mean a subset $C$ of
 $K^n$ which is stable under the $\Lambda$-action,
 that is, $\Lambda\cdot C \subset C$ (note that this implies that $\Lambda\cdot C = C$).
More generally, if $x\in K^n$,
by a
$\Lambda$-cone with origin $x$  we mean a subset $C$ of
 $K^n $ such that $C-x$ is stable under the
$\Lambda$-action, where $C-x=\{t\in K^n\mid t+x\in C\}$.
By a local $\Lambda$-cone with origin $x$, we mean a set of the form
$C \cap B (x, n)$, with $C$ a $\Lambda$-cone with origin $x$
and $n$ in $\NN$.

In Lemma \ref{cones1} we describe all possible $\Lambda$-cones which are subsets of $K$, which turns out to be very similar to the real situation. In section \ref{sec:localc} we will show that definable sets in dimension $1$ locally look like local $\Lambda$-cones (Lemma \ref{coneone}), and similarly in families of definable subsets of $K$ (Corollary \ref{corcone1}). From \ref{sec:tangentc} on we will define and study tangent cones and related objects, and formulate one of our main results on the relation between local densities of definable sets and of their tangent cones, viewed with multiplicities (the $p$-adic analogue of Thie's result), see Theorem \ref{mt}.

\begin{lem}\label{cones1}
Let $C\subset K$ be a set. Then $C$ is a $\Lambda$-cone for some $\Lambda$ in $\cD$ if and only if it is either the empty set or it is a finite disjoint union of sets of the form $\lambda P_n$ with $n>0$ and $\lambda\in K$. Hence, any cone $C\subset K$ is a definable set.
\end{lem}
\begin{proof}
Clearly the empty set is a $\Lambda$-cone for all $\Lambda$ and $\lambda\cdot P_n$ is a $\Lambda$-cone for $\Lambda\subset P_n$, and similarly for their finite unions. Now let $C$ be a nonempty $\Lambda$-cone for some $\Lambda$ in $\cD$. Either $C=\{0\}$ and we are done, or, up to replacing $C$ by $tC$ for some nonzero $t\in K$, we may suppose that $1\in C$. But then $\Lambda\subset C$ and $C\setminus \Lambda$ is still a $\Lambda$-cone. Since the index of $\Lambda$ in $K^\times $ is finite, it follows by a finite process  that $C$ consists of a finite union of sets of the form $\mu\Lambda$ with $\mu\in K$. It remains to prove that $\Lambda$ itself is a finite disjoint union of sets of the form $\lambda P_n$,
for some $n\in\NN$ and some $\lambda\in K^\times$. Since $\Lambda$ has finite index in $K^\times$, it must contain an open neighborhood $U=1+\cM_K^\ell$ of $1$ for some $\ell>0$ and with $\cM_K$ the maximal ideal of $R$. Let $\pi_K$ be a uniformizer of $R$. Since $\Lambda$ has finite index in $K^\times$, there exists $n_1>0$ such that $\pi_K^{n_1}$ lies in $\Lambda$. Now let $n$ be a big enough multiple of $n_1$ such that $t^n\in U$ for all $t\in R^\times$. Then clearly $P_n\subset \Lambda$   and we are done since $P_n$ has finite index in $K^\times$ and hence also in $\Lambda$.
\end{proof}

\subsection{Local conic structure of definable sets}\label{sec:localc}

Let $X$ be a definable subset of $K^n$ and let $x$ be a point in $K^n$. We denote by $\pi_x : K^n \setminus \{x\}\rightarrow
\PP^{n - 1} (K)$ the function which to a point $z \not=x$ assigns
the line containing $x$ and $z$. That is, for $x=0$, $\pi_0 : K^n \setminus \{0\}\rightarrow
\PP^{n - 1} (K)$ is the natural projection, and, for nonzero $x$, the map $\pi_x$ is the composition of $\pi_0$ with the translation $K^n\setminus \{x\}\to K^n \setminus \{0\}: y\mapsto y-x$.
Furthermore we denote by
$$
\pi_x^X : X \setminus \{x\}
\rightarrow \pi_x (X \setminus \{x\})
$$
 the restriction of $\pi_x$ to $X
\setminus \{x\}$.
\begin{lem}\label{coneone}
Let $Y$ be a definable subset of $K$. Then there exist $\Lambda$
in $\cD$ and a definable function $\gamma: K \to \NN$ such that $Y \cap B(y,\gamma(y))$ is a local $\Lambda$-cone with origin
$y$, for all $y\in K$. If one writes $Y$ as a finite disjoint union of cells $Y_i$ with cosets $\lambda_i P_{n_i}$, then one can take $\Lambda=P_N$
with $N=\lcm( n_i)_i$. 
\end{lem}
\begin{proof}
The definability of $\gamma$ is not an issue by the definability of the conditions of being a local $\Lambda$-cone with origin $y$ and so on.
 By definition, a finite union of local $\Lambda$-cones is again a local $\Lambda$-cone.
Hence, up to a finite partition using cell decomposition, we may suppose that $Y$ is a cell. Thus, $Y$ is of the form
 $$
Y= \{t\in K\mid |\alpha|\sq_{1} |t-c|\sq_{2} |\beta|,\
  t-c\in \lambda P_{n}\},
 $$
for some constants $n>0$, $\lambda,c$ in $K$, $\alpha,\beta$ in
$K^\times$, and $\square_{i}$ either $<$ or no condition.
  Up to a transformation $t\mapsto t-c$ we may suppose that $c=0$. We may exclude the trivial case that $Y$ is a singleton, that is, we may suppose that $\lambda\not=0$.
Then $Y$ is open, and moreover, $Y$ is closed if and only if $\sq_2$ is $<$. In the case that $\sq_2$ is no condition, then the closure of $Y$ equals $Y\cup\{0\}$.
Take $y\in K$. If $y$ lies outside the closure of $Y$, then $Y\cap B(y,n)$ is empty for sufficiently large $n$, and the empty set is a $\Lambda$-cone for any $\Lambda$ in $\cD$. Also, if $y$ lies in the interior of $Y$, then $Y\cap B(y,n)$ is a ball around $y$ for sufficiently large $n$, and hence it is a local $\Lambda$-cone with origin $y$, for any $\Lambda$ in $\cD$. Finally, if $y=0$ and $y$ lies in the closure of $Y$, then $Y\cap B(y,n)= \lambda P_n \cap B(y,n)$ for sufficiently large $n$, which is clearly a local $\Lambda$-cone with origin $y$ for any $\Lambda$ contained in $P_n$.
\end{proof}

The following two corollaries of Lemma \ref{coneone} are immediate.
\begin{cor}\label{corcone1}
Let $Y$ be a definable subset of $K^{m+1}$. For each $x\in K^m$ write $Y_x$ for $\{t\in K\mid (x,t)\in Y\}$. Then there exist $\Lambda$
in $\cD$ and a definable function $\gamma: K^{m+1} \to \NN$ such that $Y_x \cap B(t,\gamma(x,t))$ is a local $\Lambda$-cone with origin
$t$, for all $(x,t)\in K^{m+1}$. If one writes $Y$ as a finite disjoint union of cells $Y_i$ with cosets $\lambda_i P_{n_i}$, then one can take $\Lambda=P_N$
with $N=\lcm( n_i)_i$.
\end{cor}
We will most often use the following variant of Corollary \ref{corcone1}, which can be proved by working on affine charts.
\begin{cor}\label{localconic}
Let $X$ be a definable subset of $K^n$ and let $x$ be a point in
$K^n$. Then there exist a definable function $\alpha_x :
\PP^{n - 1} (K) \rightarrow \NN$, that is, $\alpha_x$ is definable on each
affine chart of  $\PP^{n - 1} (K)$, and a group $\Lambda$  in
$\cD$ such that
$$
(\pi_x^X)^{-1} (\ell) \cap B (x, \alpha_x (\ell))
$$
 is a local
$\Lambda$-cone with origin $x$ for every $\ell$ in $\PP^{n - 1} (K)$. Moreover,
$\Lambda$ can be taken independently of
$x$, and one can ensure that $(x,\ell)\mapsto \alpha_x(\ell)$ is a definable function from $K^n\times \PP^{n - 1} (K)$ to $\NN$.
\end{cor}

We shall call a subgroup $\Lambda$ in $\cD$ satisfying the first
condition in Corollary \ref{localconic} adapted to $(X, x)$, and if moreover $\Lambda$ is adapted to $(X,x)$ for all $x\in K^n$, then we call $\Lambda$ adapted to $X$.

\subsection{Tangent cones}\label{sec:tangentc}
Now, if $X$ is a definable subset of $K^n$,  $x$ a point of $K^n$, and $\Lambda$ in $\cD$,
we define
the tangent $\Lambda$-cone to $X$ at $x$ as
\begin{equation*}
\begin{split}
C_x^{\Lambda} (X)
:=
\Bigl\{u \in K^n
; (\forall i > 0)
(\exists z \in X)
(\exists \lambda \in \Lambda) \, \mbox{ such that } \\
\ord (z - x)>i \,  \mbox{ and } \, \ord (\lambda (z - x) - u)  >i
\Bigr\}.
\end{split}
\end{equation*}
By construction $C_x^{\Lambda} (X)$ is a
closed,
definable, $\Lambda$-cone, and, for any $n\in \NN$,
$C_x^{\Lambda} (X) = C_x^{\Lambda} (X\cap B(x,n))$.
Furthermore,
for definable $X,Y\subset K^n$ and for $\Lambda' \subset \Lambda$ in $\cD$, one has
$$C_x^{\Lambda} (X\cup Y)=
C_x^{\Lambda} (X)\cup C_x^{\Lambda} (Y),$$
$$ C_x^{\Lambda}(\overline {X})
=C_x^{\Lambda}(X),$$
$$
C_x^{\Lambda'} (X) \subset C_x^{\Lambda} (X).
$$
Although the previous inclusion might be strict, $\dim(C_x^\Lambda(X))$ does not depend on $\Lambda\in \cal D$ by Lemma \ref{dim}.
We comment some more on the previous inclusion in the following remarks.

\begin{remark}\label{remarque sur les cones locaux}
Let  $X$ be a local $\Lambda$-cone with origin $x$ in $K^n$. Thus, there exist $n$ in $\NN$ and $C$ a $\Lambda$-cone with origin $x$
such that
$X=C\cap B(x,n)$. In this case for any $\Lambda' \subset \Lambda\in
\cD$, one has
$$C_x^{\Lambda'} (X) =C_x^{\Lambda} (X)  \ (=C, \hbox{ when C is closed).} $$
 Indeed,
since $X=C\cap B(0,n)$,
we have $C_x^{\Lambda'} (X)=C_x^{\Lambda'} (C)$. But
we also have $C_x^{\Lambda'} (C)=C_x^{\Lambda} (C)$ ($=C$, when
$C$ is closed).

We indicate why $C_x^{\Lambda} (C)\subset C_x^{\Lambda'} (C)$.
Assuming $x=0$ for simplicity,
let $u\in C_0^{\Lambda} (C)$ and $i\in \NN$, $z\in C$, $\lambda
\in \Lambda$, with $\ord(z)>i$ and $\ord(\lambda z-u)>i$.
We have $z\in C$ and thus $\lambda z\in C$. Now let
$\lambda'\in \Lambda'$ small enough to ensure that $\ord(
\lambda'\lambda z)>i$.  Then denoting $z'=\lambda'\lambda z$, one has
$z' \in C$. From
$\ord(z')>i$ and $\ord((1/\lambda')z'-u)>i$, we see  that
$u \in C_0^{\Lambda'} (C)$.

Finally we indicate why $C_x^{\Lambda} (C)=C$, when $C$ is closed.
As $C$ is stable by the
$\Lambda$-action, the inclusion $C\subset  C_x^{\Lambda} (C)$ is obvious.
On the other hand, assuming again $x=0$, if $u\in C_x^{\Lambda} (C)$,
for all $i\in \NN$, there exist $z\in C$ and $\lambda\in \Lambda$
such that $\ord(z)>i$ and $\ord(\lambda z-u)>i$. We can then construct
a sequence of points $\lambda z\in C$ with limit $u$, this shows that
$u\in C$, since $C$ is closed.
\end{remark}

\begin{remark}
When
$X$ is a definable subset of $K$ and  $x$   a point of
$K$, by Lemma \ref{coneone}, $X$ is a local $\Lambda$-cone
at $x$ with origin $x$, for some $\Lambda\in \cD$. By the above remark,
for every $\Lambda' \subset \Lambda$, and still in the one dimensional case that $X\subset K$, one has $C_x^{\Lambda'} (X) =
C_x^{\Lambda} (X)$.

We cannot expect in general that for $X$ a definable subset of  $K^n$, $n>1$,
$X$ is a local $\Lambda$-cone for some $\Lambda\in \cD$, but one
 may at least ask, as it is the case for $n=1$, whether
the stability property:
``{\sl there exists $\Lambda\in \cD$ such that for any $\Lambda'\in \cD$,
$\Lambda'\subset \Lambda$, one has  $C_x^{\Lambda'} (X) =
C_x^{\Lambda} (X)$}" still holds for $n>1$. The answer to that question is yes,
as we shall show in
 Theorem \ref{distinguished cone}.
\end{remark}

\subsection{More on $\varepsilon$-analytic functions}

The following is the $p$-adic analogue of Proposition 1.7 of
\cite{KR}.

\begin{prop}\label{prop1.7KR}
Let $f:U\to K^{n-d}$ be a definable $\varepsilon$-analytic function on a
nonempty open subset $U$ of $K^{d}$, $0\leq d\leq n$. Let $\Gamma$
be the graph of $f$ and let $z$ be in $\overline\Gamma$. Then, for
any group $\Lambda$  in $\cD$
$$
C_z^{\Lambda} (\Gamma) \subset \{(x,y)\in K^{d}\times K^{n-d}\mid
|y|\leq \varepsilon |x|\}.
$$
\end{prop}
\begin{proof}
We may suppose that $z=0$. 
Choose $\Lambda$  in $\cD$. Since
$$
C_0^{\Lambda} (\Gamma) \subset C_0^{K^\times} (\Gamma),
$$
by definition of $C_0^{\Lambda}(\cdot)$, we may assume that
$\Lambda=K^\times$. We may also suppose that $\varepsilon=1$, after
rescaling. Suppose by contradiction that there is $(x_0,y_0)$ in
$C_0^{K^\times} (\Gamma)$ with $|y_0|>|x_0|+\delta$ for some $\delta>0$.
Let $\Gamma'$ be the intersection of $\Gamma$ with the open
subset $\{(x,y)\in K^{d+(n-d)}\mid |y|>|x|+\delta\}$. By our assumption on
$(x_0,y_0)$ and by the definition of $C_0^{K^\times} (\Gamma)$, the
set $\Gamma'$ is nonempty and $0$ lies in $\overline\Gamma'\setminus
\Gamma'$. Apply the Curve Selection Lemma \ref{curve} to the set
$\Gamma'$ and the point $0$. This way we find power series $g_i$
over $K$ in one variable for $i=1,\ldots,n$, converging on $R$, such
that $g(0)=0$ and $g(R\setminus \{0\})\subset
\Gamma'\setminus\{0\}$. But this is in contradiction with Lemma
\ref{lemWhitney}. Indeed, the tangent line $\ell'_r$ at $r\not=0$ is
of the form $g(r) + K\cdot t_r$ with some $t_r\in K^n$ satisfying
$|y(t_r)|\leq |x(t_r)|$ by $\varepsilon$-analyticity of $f$ and the
chain rule for differentiation, where
$x(t_r)=(t_{r1},\ldots,t_{rd})$ and
$y(t_r)=(t_{rd+1},\ldots,t_{rn})$. Hence, the limit $\ell'_0$ of the
$\ell'_r$ for $0\not=r\ \to 0$ is of the same form $g(0) + K\cdot t_0$
for some $t_0\in K^n$ with $|y(t_0)|\leq |x(t_0)|$. On the other
hand, the line $\ell_r$ for $r\not=0$ connecting $g(0)$ with $g(r)$
is of the form $g(r) + K\cdot u_r$ with some $u_r\in K^n$ satisfying
$|y(u_r)| > |x(u_r)|+\delta$. Hence, the limit line $\ell_0$ of
the $\ell_r$ for $r\to 0$ has the same description, which
contradicts Lemma \ref{lemWhitney} and the description of $\ell'_0$.

\end{proof}

\begin{cor}\label{cor1.8KR}
With the data and the notation of Proposition \ref{prop1.7KR}, let $x$ be in
$\overline U$. Then there are only finitely many points in
$\overline\Gamma$ which project to $x$ under the coordinate projection
$K^{d}\times K^{n-d}\to K^d$.
\end{cor}
\begin{proof}
Suppose by contradiction that there are infinitely many such points. Then
the dimension of $\overline\Gamma\cap (\{x\}\times K^{n-d})$ is
$>0$. Thus, there exists $z\in \overline\Gamma$ such that
$C_z^\Lambda( \overline\Gamma\cap (\{x\}\times K^{n-d}) )$ is
of dimension $>0$, which is in contradiction to Proposition \ref{prop1.7KR}.
\end{proof}

\subsection{Deformation to the tangent cone}
Let
$X$ be a definable subset of $K^n$ and let $x$ be a point of $K^n$.
Fix a subgroup $\Lambda$ in $\cD$.
We consider the definable set
$\cD (X, x, \Lambda)$
in $K^n \times \Lambda$ defined as
$$
\cD (X, x, \Lambda)
:=
\Bigl\{ (z, \lambda)
;
x + \lambda z \in X
\Bigr\}
$$
and its closure
$$
\overline{\cD (X, x, \Lambda)}
$$
in
$K^n \times K$. In $\overline{\cD (X, x, \Lambda)}$ one finds back the cone $C_x^{\Lambda} (X)$. Indeed,
one has $ \overline{\cD (X, x, \Lambda)} \cap (K^n\times \{0\}) = C_x^{\Lambda} (X) \times \{0\}$, which we identify with $C_x^{\Lambda} (X)$.

\begin{lem}\label{dim}
If $X$ is of dimension $d$, then $\cD (X, x, \Lambda)$ is of
dimension $d+1$ and $C_x^{\Lambda} (X)$ is of dimension $\leq d$. Moreover, $\dim(C_x^\Lambda(X))$ does not depend on the choice of $\Lambda\in \cal D$.
\end{lem}
\begin{proof}
We may suppose that $X$ is nonempty. Consider the projection
\begin{equation*}
p :
\begin{cases}
\cD (X, x, \Lambda)
\longrightarrow X  \\
(z,\lambda)\longmapsto x+\lambda z.
\end{cases}\end{equation*}
Since $p$ is surjective and has fibers of dimension $1$, we get that $\cD (X, x, \Lambda)$ is of dimension $d+1$.
 The cone $C_x^{\Lambda} (X)$ is contained in
$\{0\}\cup(\overline{\cD (X, x, \Lambda)}\setminus \cD (X, x,
\Lambda))$. Hence, $C_x^{\Lambda} (X)$ is of dimension $\leq d$. The last statement follows from
$$
C_x^{\Lambda} (X) = \bigcup_i \mu_i C_x^{\Lambda'} (X),
$$
whenever $\Lambda'\subset \Lambda$ is in $\cD$ and when one writes $\Lambda$ as a finite union of cosets $\bigcup_i \mu_i \Lambda'$ of $\Lambda'$ in $\Lambda$.
\end{proof}

\subsection{Multiplicities on the tangent cones}

Let $X$ be a definable subset of $K^n$ of dimension $d$, let $x$
be a point of $K^n$, and let $\Lambda$ be in $\cD$.  To each point $z$ on the cone $C_x^{\Lambda} (X)$ we will associate a rational number $SC_x^{\Lambda} (X)(z)$, called the multiplicity of $(X,x)$ at $z$ with respect to $\Lambda$.


Define the function
$$SC_x^{\Lambda} (X) : C_x^{\Lambda} (X) \to \QQ $$
 as the function sending $z$ to
$$
[K^{\times} : \Lambda] \, \Theta_{d + 1}(\11_{\cD (X, x, \Lambda)})(z,0),
$$
with $[K^{\times} : \Lambda]$ the index of
$\Lambda$ in $K^\times$, and with $\11_{\cD (X, x, \Lambda)}$ the characteristic function of $\cD (X, x, \Lambda)$. The function $SC_x^{\Lambda} (X)$ is called the specialization of $X$ at $x$ with respect to $\Lambda$.


The following lemma gives an  indication that $SC_x^{\Lambda} (X)$ captures much local information of $(X,x)$; this principle will find a strong and precise form in Theorem \ref{mt} below.
\begin{lem}\label{support}
The function $SC_x^{\Lambda} (X)$ lies in $\cC( C_x^{\Lambda} (X))$.
Moreover
$$
\Theta_{d + 1}(\11_{\cD (X, x, \Lambda)})(z,0)=0
$$
for $z$ outside $C_x^{\Lambda} (X)$.
\end{lem}
\begin{proof} The function $\11_{\cD (X, x, \Lambda)}$ is in $\cC(K^n \times K)$ since $\cD (X, x, \Lambda)$ is a definable set, and thus, also $\Theta_{d + 1}(\11_{\cD (X, x, \Lambda)})$ lies in $\cC(K^n \times K)$.
For definable sets $A\subset B$, the restriction of a function in $\cC(B)$ to $A$ automatically lies in $\cC(A)$, hence,  $SC_x^{\Lambda} (X)$ lies in $\cC( C_x^{\Lambda} (X))$. The second statement follows from the fact that the support of $\Theta_{d + 1}(\11_{\cD (X, x, \Lambda)})$ is contained in the closure of $\cD (X, x, \Lambda)$, which is contained in $(K^n\times K^\times) \cup  (C_x^{\Lambda} (X) \times \{0\})$.
\end{proof}

\medskip

More generally, if $\varphi$ is a function in $\cC (X)$ which is
bounded near $x$, we define the specialization $\nu_x^{\Lambda}
(\varphi)$ of $\varphi$ at $x$ with respect to $\Lambda$ in the following way.
First define a function $\psi$ on $K^n \times K$ by $\psi (z, \lambda) := \varphi (x + \lambda
z)$ on  $\cD (X, x, \Lambda)$ and by zero elsewhere. Then one defines the function
$$\nu_x^{\Lambda} (\varphi):C_x^{\Lambda} (X) \to \QQ$$
as the function sending $z$ to
$[K^{\times} : \Lambda] \, \Theta_{d + 1}(\psi)(z,0)$. Note that, similarly as in Lemma \ref{support},  $\nu_x^{\Lambda} (\varphi)$ lies in $\cC
(C_x^{\Lambda} (X))$ and that  $\Theta_{d + 1}(\psi)(z,0)=0$ for $z$ outside $C_x^{\Lambda} (X)$.
We recover $SC_x^{\Lambda} (X)$ since $\nu_x^{\Lambda} (\11_X)= SC_x^{\Lambda} (X) $.

\medskip
The following result, which will be proved in
section \ref{secmt},
states that the local density can be computed on the tangent cone with multiplicities, for $\Lambda$ small enough.

\begin{theorem}\label{mt} Let
$X$ be a definable subset of $K^n$ and let $x$ be a point of $K^n$.
For $\Lambda$ small enough
$$
\Theta_d (X) (x)
=
\Theta_d (SC_x^{\Lambda} (X)) (0).
$$
More generally,
let $\varphi$ be a function in
$\cC (X)$ which is bounded near $x$. For $\Lambda$ small enough
$$
\Theta_d (\varphi) (x)
=
\Theta_d (\nu_x^{\Lambda} (\varphi)) (0).
$$
\end{theorem}


\section{Existence of $(w_f)$-regular stratifications}\label{swf}

\subsection{}In his study of stability
of the topological type of mappings,
R. Thom introduced the regularity condition $(a_f)$
in \cite{Thom}, p. 274,
as a relative version of condition $(a)$ of Whitney. The existence of $(a_f)$-regular stratifications was proved in the complex analytic case by H. Hironaka in \cite{Hi}
(Corollary 1, Section 5) using resolution of singularities, under
the assumption ``sans \'eclatement" which is always satisfied for functions.
One can find proofs of the existence of $(a_f)$ stratifications in the
real subanalytic case in \cite{KR}, where the Puiseux Theorem
with parameters of Paw\l ucki (see \cite{Paw}) is used, and for o-minimal
structures on the field of real numbers in \cite{Loi1}.

The stronger condition $(w_f)$, the relative version of the
so-called
condition $(w)$ of Verdier (see \cite{Verdier}),
was studied in the complex setting, for instance,  in \cite{HeMeSa}.
In the real
subanalytic
setting, it has been proved that $(w_f)$ stratifications
exist
by K. Bekka in \cite{bekka}, K. Kurdyka and A. Parusi\'nski in \cite{KurPar}
using Puiseux Theorem with parameters, and finally by Ta L\^e Loi in
\cite{Loi2} for definable functions in some o-minimal structures
over the real field (the
o-minimal
 structure has to be polynomially bounded for the
existence of $(w_f)$-regular stratifications, but need not to be so for
the existence of $(a_f)$-regular stratifications).

\subsection{}\label{4.2}
Let us now  recall the definitions of $(w_f)$ and $(a_f)$-regular
stratifications.
Let
$X$ be a definable subset of $K^n$, and let
 $(X^j)_{j\in \{1, \cdots, k\}}$ be a finite,
definable and analytic stratification
of $X$
satisfying the so-called
frontier condition
$$X^i\cap \overline{X^j}\not=\emptyset \Longrightarrow X^i\subset
\overline{X^j},$$
where definable and analytic means that the strata $X^j$ are definable, $K$-analytic manifolds.
Let $S$ be a definable subset of $K$
and let $f:X\to S$ be a definable continuous mapping such that for any
$j\in \{1, \cdots, k\}$, $f_{\vert X^j}$ is
analytic and of constant rank
(being $0$ or $1$).
For $j\in \{1, \cdots, k\}$ and
$x\in X^j$, let us denote by $T_xX^j_f$ the tangent space at $x$
of the fiber $ f^{-1}_{\vert X^j}(f(x))$ of $f_{\vert X^j}$.
Then one says
that the pair of strata $(X^i,X^j)$  satisfies condition $(a_f)$
at a point $x_0\in X^i\subset \overline{X^j}$ if and only if
for any sequence $(x_r)_{r\in \NN \setminus\{0\} }$ of points of $X^j$ converging to
$x_0$ and such that the sequence
$( T_{x_r}X^j_f )_{r\in \NN \setminus\{0\} }$ converges in the appropriate Grassmann manifold,
one has
\begin{equation}\tag{$a_f$}
 \lim_{r\to \infty} \delta(T_{x_0}X^i_f,T_{x_r}X^j_f)= 0,
 \end{equation}
where $\delta(\cdot,\cdot) $ is a natural distance between linear subspaces of $K^n$ as defined below.
Further,
one says
that the pair $(X^i,X^j)$ of strata satisfies condition $(w_f)$ at $x_0$
if and only if
there exist a constant $C$ and a neighborhood ${\cal W}_{x_0}$ of $x_0$ in
$K^n$, such that for any $x\in {\cal W}_{x_0}\cap X^i$
 and any $y\in {\cal W}_{x_0}\cap X^j$, one has
\begin{equation}\tag{$w_f$}
  \delta(T_xX^i_f,T_{y}X^j_f)\le C\cdot \vert x-y\vert .
\end{equation}
In both definitions, $\delta(V,V')$ denotes the distance
between
two linear subspaces $V$ and $V'$ of $K^n$
such that $\dim(V)\le \dim(V')$,
and is defined by
$$\delta(V,V')=\sup_{v\in V, \vert v\vert =1}
\{\inf_{v'\in V', \vert v'\vert=1}\vert v-v' \vert \}=
\sup_{v\in V, \vert v\vert =1}\hbox{dist}(v,S^{V'}(0,1)),$$
with $S^{V'}(0,1)$ the unit sphere
around $0$
of $V'$.

\begin{remark}\label{angle}
We have $\delta(V,V')=0$ if and only if $V\subset V'$ and
for any $V''\subset V'$
such that $\dim(V)\le \dim(V'')$,
$\delta(V,V'')\ge \delta(V,V')$.
\end{remark}

One says that the stratification $(X^j)_{j\in \{1, \cdots, k\}}$
is $(a_f)$-regular, respectively $(w_f)$-regular
if any pair $(X^i,X^j)$ of strata is $(a_f)$-regular,
respectively $(w_f)$-regular at any point of $X^i$.
And finally one says that the stratification
$(X^j)_{j\in \{1, \cdots, k\}}$
is $(a)$-regular, respectively $(w)$-regular, if it is
$(a_f)$-regular, respectively $(w_f)$-regular, for $S$ a point in $K$.

One starts the proof of the existence of $w_f$-regular stratifications
with the key Lemma \ref{pli} (see \cite{Loi2}, Lemma 1.8 for
its real version).
But before stating this lemma, let us introduce as in
\cite{CCL} (Definition 3.9)
 the notion of jacobian
property for definable functions
and recall from \cite{CCL} that this
property is
in a sense
a generic one
(see Proposition 3.10 of \cite{CCL} or Proposition \ref{jacprop} below).
This will be used in the proof of Lemma \ref{pli}.

\begin{definition}
Let $F:B\to B'$ a definable function with $B,B'\subset K$.
 We say that $F$ has the jacobian property if the following conditions hold all together:
\vskip2mm
$\rm{(i)}$ $F$ is a bijection and $B,B'$ are balls,
\vskip1mm
$\rm{(ii)}$ $F$ is $C^1$ on $B$,
\vskip2mm
$\rm{(iii)}$ $\displaystyle
\ord \Bigl(\frac{\partial F}{\partial x}\Bigr):B\to \ZZ$ is constant (and finite) on
$B$,
\vskip2mm
$\rm{(iv)}$ for all $x,y\in B$ with $x\not=y$, one has
$$ \ord\Bigl(\frac{\partial F}{\partial x}\Bigr)+\ord(x-y)
=\ord(F(x)-F(y)).$$
\end{definition}

It is proved
in a much more general setting in \cite{CLip}, Theorem 6.3.7,
that the jacobian property is generic for
definable mappings, which in our setting gives the following statement.

\begin{prop}\label{jacprop}
Let $Y\subset K^m$ and $X\subset K\times Y$ be definable sets for some $m\in \NN$.
 Let $M:X\to K$ be definable.
 Then there exists a finite partition of $X$ into definable subsets
$X_k$ such that for each $y\in Y$, the restriction
$M(\cdot, y): x_1\mapsto M(x_1,y) $  of $M$ to $\{x_1\in K; (x_1,y)\in X_k\}$
is either injective or constant.

Let us then assume, for simplicity, that on $X$, $M(\cdot,y)$ is injective.
Then there exists a finite partition of $X$ into cells $A_k$ over $Y$
such that for each $y\in Y$ and each ball $B$
such that $B\times \{y\}$ is contained in $A_k$,
there
is a
(unique) ball $B'$ such that the map $M_{\vert B}:B\to B':  x_1\mapsto M(x_1,y)\in B'$
has the jacobian property.
\end{prop}

Now we state and prove the key lemma used in the proof of Theorem \ref{wf}.
\begin{lem}\label{pli}
Let  $M:\Omega\to K$ be a definable and differentiable
function on an open
definable subset $\Omega$ of $K^m\times K$ for some $m\geq 0$.
Assume that $\overline\Omega \cap (K^m\times \{0\})$ has a
nonempty
 interior $U$ in $K^m$. Assume furthermore that $M$ is bounded
on $\Omega$. Then there exist a nonempty open definable subset
$V\subset U$ in $K^m$, an integer $\alpha>0$  and a constant $d\in K^\times$, such that
for all $x\in V$ and all $t$ with $\ord t > \alpha$ and $(x,t)\in  \Omega$
$$ \Vert D_xM_{(x,t)}\Vert \le \vert d  \vert\ . $$
\end{lem}
In the above lemma and later on, $D_x M_{(x,t)}$ means $(\partial M(x,t)/\partial x_1,\ldots,\partial M(x,t)/\partial x_m )$, and analogously, $D_{x_1} M_{(x,t)}$ means $\partial M(x,t)/\partial x_1$ and so on.

\begin{proof}
Let us denote $(x,t)=(x_1,\cdots, x_m,t)=(x_1,y)$ the standard
coordinates on $K^m\times K=K\times K^m$,
where $y=(x_2,\cdots, x_m,t)$. (We will apply cell decomposition and related results sometimes with $x_1$ and sometimes with $t$ as special variable.)
By the Cell Decomposition Theorem (with special variable $t$) we can finitely partition $\Omega$ such that on each part $A$ such that $\overline A$ has nonzero intersection with $K^m\times \{0\}$ one has
$$
\vert D_{x_i}M_{(x,t)}\vert = \Vert D_{x}M_{(x,t)}\Vert =  \vert c(x) \vert \cdot \vert \lambda t\vert^{a} $$
for some $a\in \QQ$, some $\lambda\in K^\times $, some $i\in\{1,\ldots,m\}$, and some definable function $c$. If all these exponents $a$ are nonnegative, then we are done since the $|c|$, as well as the boundary functions bounding $|t|$ from below in the cell descriptions are constant on small enough open subsets $V\subset U$.
Let us assume that a particular $a$ is negative, say on a cell where $\vert D_{x_1}M_{(x,t)}\vert = \Vert D_{x}M_{(x,t)}\Vert$. By Proposition \ref{jacprop} (with special variable $x_1$) applied to
$ M  :    (x_1,y)
\mapsto M(x_1,y) $
there exists a finite number of cells $A_k$
partitioning
 $\Omega$
such that for each $y\in K^m$ and each ball
$B$ with  $B\times \{y\}\subset A_k$, $B\ni x_1\mapsto M(x_1,y)$ has the
jacobian property.
We necessarily have one of these cells $A_k$ such that
$\overline{A_k} \cap (K^m\times \{0\})$
has nonempty
interior
in $K^m$.
We may assume by the Cell Decomposition Theorem (with special variable $t$) that
$ A_k$
 contains a subset $B_1\times B'\times W$ with
$B_1$ an open ball in the $x_1$ line, $B'$ a Cartesian product of $m-1$ balls  and
$W$ an open definable subset of
$K^\times$
 such that $0\in \overline W$.
Then, for any $y=(x_2,\cdots, x_m,t)\in B'\times W$, by
the jacobian property,
the one-dimensional volume $\mu_1(M(B_1\times \{y \}))$ equals $\mu_1(B_1)\cdot
|c(x)|\cdot \vert t \vert^{a}$. Considering that $t$ can approach $0$ while $|c(x)|$ stays constant and that $M$ is a bounded mapping,  this is a contradiction.
\end{proof}

We can finally prove our result concerning $(w_f)$-regular stratifications.

\begin{theorem}\label{wf}
 Let $X$ be a definable subset of $K^n$, $S$ a definable subset of $K$ and
 $f:X\to S$ a definable continuous function. Then there exists a \textup{(}finite\textup{)} analytic
definable
stratification of $X$ which is $(w_f)$-regular. In particular definable
subsets of $K^n$ also admit $(a_f)$, $(w)$ and $(a)$-regular
definable stratifications.
 \end{theorem}
\begin{proof} We proceed similarly as in \cite{Loi2}.
Let $(X^j)_{j\in \{1, \cdots, k\}}$ be an analytic and definable stratification
of $X$ such that the $f_{\vert X^j}$ are analytic and such that the rank of $f_{\vert X^j}$ is constant for all
$j\in \{1, \cdots, k\}$.
The set $w_f(X^i,X^j)$ of points $x\in X^i$ at which
the pair $(X^i,X^j)$ is $(w_f)$-regular being a definable set,
we have to show that this set is dense in $X^i$. Let us assume
that the contrary holds, that is, the set $w'_f(X^i,X^j)$ of points
of $X^i$ at which the pair $(X^i,X^j)$ is not $(w_f)$-regular is
dense in (a nonempty open subset of) $X^i$, and let us obtain a contradiction.
Up to replacing $X^i$ by a nonempty subset of $X^i$ and by the definability of $w'_f(X^i,X^j)$, we may suppose that $w'_f(X^i,X^j)$ equals $X^i$.

As the condition
$(w_f)$ is invariant
under differentiable
transformations of $K^n$ with Lipschitz
continuous
derivative
and up to replacing $X^i$ by a nonempty open subset of $X^i$,
we may assume
that $X^i$ is an open definable subset of $K^m\times \{0\}^{n-m}$.
(The latter transformation exists by Cell Decomposition, after shrinking $X^i$ if necessary.)
Up to replacing $X^i$ by a nonempty subset,
we may also assume that $f_{\vert X^i}$ is constant, equal to $0$
for simplicity.
Indeed,
since $w'_f(X^i,X^j)$ equals $X^i$,
we can replace $X^i$ by $f_{\vert X^i}^{-1}(a)$ with $a\in f(X^i)$, since the pair
$(f_{\vert X^i}^{-1}(a), X^j)$ is $(w_f)$-regular at none of the points of $f_{\vert X^i}^{-1}(a)$.

\medskip
Now we have two cases to consider:

\subsection*{Case 1: $f_{\vert X^j}$  is constant (in a neighborhood of $X^i$)}
Then condition $(w_f)$ is condition $(w)$.
We  proceed as follows.

Write  $X^i = U\times \{0\}^{n-m} $ with $U$ open in $K^m$.
By the Cell Decomposition Theorem and
the existence of definable choice functions and up to making $U$ smaller, there exists a
definable
$C^1$ function $\rho: U\times C \to X^j$ (called a $C^1$ wing in $X^j$ in \cite{Loi2}),
where $C$ is a one dimensional cell in
$K^\times$ with $0\in\overline C$,
such that $\rho(x,t)=(x,r(x,t))$ and $\vert r(x,t)\vert < \vert t \vert$,
and furthermore, $w'_f(X^i,X^j)$ being assumed
equal to
$X^i$, we may ask
that for all $x,t$
$$
\frac{\delta( K^m\times\{0\}^{n-m}, T_{\rho(x,t)}X^j)}
{\vert r(x,t)\vert}   \geq |t|^{-1}.
$$
By Remark \ref{angle}, one then has
$$
\frac{\Vert D_x r_{(x,t)}\Vert }{\vert r(x,t)\vert }
\ge \frac{\delta( K^m\times\{0\}^{n-m}, T_{\rho(x,t)}X^j)}
{\vert r(x,t)\vert}.
$$
By Cell Decomposition and up to replacing the function $(x,t)\mapsto r(x,t)$ by $(x,t)\mapsto r(x,t^s)$ for some integer $s>0$, we may moreover assume that on $U\times C$
$$
|r(x,t)|=|a|\cdot|t|^\ell
$$
for some integer $\ell>0$ and some $a\in K^\times$.
But when one applies Lemma \ref{pli} to $(x,t)\mapsto  r(x,t)/t^\ell$, which is a bounded
map, one finds a definable nonempty open subset $U'$  of $U$,
and $d \in K^\times$ such that for $x\in U'$ and $t\in C$ with
$\vert t \vert $ small enough,  $\Vert D_xr_{(x,t)}
\Vert / |r(x,t)| \le \vert d \vert$, a contradiction with the above two displayed inequalities.

\subsection*{ Case 2:  $f_{\vert X^j}$ has rank 1}

Write  $X^i = U\times \{0\}^{n-m} $ with $U$ open in $K^m$.
Clearly we may suppose that $f(x,y)\not =0$ for $(x,y)\in X^j$, $x\in U$. We further have that for each $x\in U$, $f(x,y)$ goes to $0$ when $y$ goes to zero with $(x,y)\in X^j$.
  Hence there exists a definable choice function $f_0:B(0,1)\to f(X^j)\cup\{0\}$ such that $f_0(t)=0$ if and only if $t=0$ and $|f_0(t)|<|t|$ for nonzero $t$. Since we assume that $w'_f(X^i,X^j)$ equals $X^i$, we may moreover assume that  for each $x\in U$ and nonzero $t$ there exists $y$ satisfying $(x,y)\in X^j$, $|y|<|t|$, $f(x,y)=f_0(t)$,  and
$$\frac{\delta(K^m\times \{0\}^{n-m},T_{(x,y)}
X^j_f)}{\vert y \vert } \geq |t|^{-1}.
$$
Up to replacing $f_0$ by $t\mapsto f_0(\lambda t^s)$ for some integer $s>0$ and nonzero $\lambda\in R$, we may suppose that $f_0$ is continuous.
 Hence, by the existence of definable choice functions  there exists  a continuous definable  map
$\varphi : U\times  B(0,1) \to K^{n-m}$ which is $C^1$ on $U\times  (B(0,1)\setminus\{0\})$ and such that, for all $x\in U$ and for all nonzero $t\in B(0,1)$, $\varphi(x,0)=0$,
$(x,\varphi(x,t))\in X^j$,
\begin{equation}\label{1}
\frac{\delta(K^m\times \{0\}^{n-m},T_{(x,\varphi(x,t))}
X^j_f)}{\vert \varphi(x,t) \vert } \geq |t|^{-1}
\end{equation}
 and
 \begin{equation}\label{111}
f(x, \varphi(x,t) )= f_0(t).
\end{equation}
  It follows by (\ref{111}) that the $m$-dimensional linear space $W$ spanned by the vectors $(0,\ldots,0,1,0,\ldots,0,\partial \varphi(x,t)/\partial x_i)$ for $i=1,\ldots,m$  is a subspace of $T_{(x,\varphi(x,t))}X^j_f$.
Combining this  with Remark \ref{angle} and  since $\Vert D_x\varphi_{(x,t)}\Vert \geq \delta(K^m\times \{0\}^{n-m} , W)$, it follows that
\begin{equation}\label{1bis}
\frac{\Vert D_x\varphi_{(x,t)}\Vert}{\vert \varphi(x,t) \vert }\ge \frac{\delta(K^m\times \{0\}^{n-m},T_{(x,\varphi(x,t))}
X^j_f)}{\vert \varphi(x,t) \vert }.
\end{equation}
By the Cell Decomposition Theorem \ref{thm:CellDecomp}, by making $U$ smaller, and up to replacing the function $(x,t)\mapsto \varphi(x,t)$ by $(x,t)\mapsto \varphi(x,bt^s)$ for some integer $s>0$ and some nonzero $b\in R$,
we may suppose we have on $U\times B(0,1)$
\begin{equation}\label{2a} \vert \varphi(x, t)   \vert= \vert a  \vert \cdot \vert
t\vert^\ell, \end{equation}
 with  $a\in K^\times$ and some integer $\ell>0$,
since $\varphi$ is  continuous and $\varphi(x,t)=0$ if and only if $t=0$.
Applying Lemma \ref{pli} to
 the bounded function $\varphi(x,t)/t^\ell$ yields a  contradiction with
$(\ref{1})$ and $(\ref{1bis})$ similarly as in case 1.
 \end{proof}

\subsection{}
Let
$X$ be a definable subset of $K^n$, and let
 $(X^j)_{j\in \{1, \cdots, k\}}$ be a finite,
definable and analytic stratification
of $X$
satisfying the frontier condition as in \ref{4.2}.
Let $X^i$ and $X^j$ be strata with
$X^i \subset \overline{X^j}$
and let $x_0 \in X^i$.
One says
$(X^i, X^j)$ satisfies condition $(b)$ at $x_0$
if for every sequences
$x_m \in X^i$, $y_m \in X^j$, both converging to
$x_0$ and
such that the line
$L_m$ containing $x_m$ and $y_m$, resp.~the tangent space
$T_{y_m} X^j$, both converge in the appropriate Grassmann manifold to
a line $L$, resp. a subspace $T$, then $L \subset T$.
Over the reals, it is well known since the seminal work of T. C. Kuo
\cite{Kuo} (in the semi-analytic case), see also \cite{Verdier} (subanalytic case)
and \cite{Loi2} (o-minimal case), that
condition $(w)$ implies condition $(b)$. Note that obviously $(w)$ does not imply $(b)$
in the real differential case and that even in the real algebraic case $(b)$ does not imply
$(w)$.
In the present setting, we have a similar result (with a similar proof):

\begin{prop}If $(X^i, X^j)$ satisfies condition $(w)$ at $x_0$, it also satisfies condition
$(b)$  at $x_0$.
\end{prop}

\begin{proof}We set $X^i = W$ and
$X^j = W'$. We may assume that
$W$ is open in
$K^r \simeq  K^r \times \{0\} \subset K^r \times K^s = K^n$
 and that
$x_0 = 0$.
Denote by $p$ the linear projection
$K^n \to K^s$.
If condition $(b)$ is not
satisfied at $0$,
then,
by condition $(w)$ at $0$ and
for some $\varepsilon >0$,
one has
 $0 \in \overline{S} \setminus S$,
with
\[
S = \{x \in W' \, \vert \, \delta (K p (x), T_x W') \geq 2 \varepsilon\}.
\]
Use the curve selection lemma \ref{curve} to find
an analytic definable function
$\varphi : U \subset K \to S$ with
$0 \in \overline{U}$ such that
$\Vert \varphi (t)\Vert \leq \vert t \vert$
for all $t$ in $U$.
Write $\varphi = (a, b)$ with $a : U \to K^r$
and $b : U \to K^s$.
We may assume that
$\Vert a' (t) \Vert$ is bounded, that
$b$ and $b'$ do not vanish, and by analyticity that
$\lim_{t \to 0} \Vert b (t) \Vert /  \Vert b' (t) \Vert = 0$.
Since
$\delta (K b'(t), K b(t)) \to 0$ for $t \to 0$
which holds by Lemma \ref{lemWhitney},
we have
$\delta (K b'(t), T_{\varphi (t)} W') \geq \varepsilon$
for $t$ small enough.
From the fact that
$\varphi' (t) = a' (t) + b' (t) \in T_{\varphi (t)} W'$,
it follows that
\begin{equation}\label{one}
\frac{\Vert a' (t)\Vert}{\Vert b'(t)\Vert} \delta (K a'(t), T_{\varphi (t)} W') \geq \varepsilon.
\end{equation}
Now, by condition
$(w)$ at $0$,
there exists  $C >0$ such that
\begin{equation}\label{two}
\delta (K a'(t), T_{\varphi (t)} W') \leq C \Vert b(t) \Vert
\end{equation}
for $t$ small enough.
It follows from (\ref{one}) and (\ref{two}),
that, for $t$ small enough,
\begin{equation}
\varepsilon  \leq C \frac{\Vert b(t)\Vert}{\Vert b'(t)\Vert}
\Vert a' (t) \Vert,
\end{equation}
which contradicts the fact that
$\Vert a' (t) \Vert$ is bounded and
$\lim_{t \to 0} \Vert b (t) \Vert /  \Vert b' (t) \Vert = 0$.
\end{proof}

\section{Proof of Theorem \ref{mt} and existence of distinguished
tangent $\Lambda$-cones}\label{secmt}
\subsection{Proof of Theorem \ref{mt}: a first reduction}
The statement we have to prove being additive, we may cut $X$ into finitely many definable pieces. Also note that we may assume all these pieces
have dimension $d$ around $x$, since pieces of dimension $<d$
contribute to zero in both sides in the equality we have to prove.
Let us  prove in this subsection that we may reduce to the case were
$\varphi  = \11_X$. Suppose that we know the result for $\varphi  = \11_X$.
For the general case we may assume, by additivity and linearity, that
$\varphi  =  (\prod_{i=1}^\ell \beta_i) \cdot q^{- \alpha}$
with $\alpha$ and the $\beta_i $ definable functions from $X$ to $\ZZ$.
Further we may assume that $\varphi\geq 0$ on $X$.
Write $X$ as a possibly infinite disjoint union parameterized by the values of $\alpha$ and the $\beta_i$. That is
$$
X=\cup_{z\in\ZZ^{\ell+1}}X_z, \mbox{ with } X_z = \{x\in X\mid (\beta_1,\ldots,\beta_\ell,\alpha)(x) = z\}.
$$
Since $\varphi$ is constant on each of the $X_z$, by linearity we find for each $z$
$$
\Theta_d (\varphi_z) (x)
=
\Theta_d (\nu_x^{\Lambda} (\varphi_z)) (0),
$$
where $\varphi_z$ is the product of $\varphi$ with the characteristic function of $X_z$. By Proposition \ref{convmon} one finds
$$
\Theta_d (\varphi) (x) =  \Theta_d (\sum_z \varphi_z) (x) = \sum_z \Theta_d (\varphi_z) (x)
$$
and similarly
$$
\Theta_d (\nu_x^{\Lambda} (\varphi)) (0) = \Theta_d (\nu_x^{\Lambda} ( \sum_z \varphi_z)) (0) = \Theta_d (\sum_z \nu_x^{\Lambda} (  \varphi_z)) (0) = \sum_z \Theta_d (\nu_x^{\Lambda} (  \varphi_z)) (0),
$$
and hence $\Theta_d (\varphi) (x) = \Theta_d (\nu_x^{\Lambda} (\varphi)) (0)$
which finishes the reduction.

\subsection{Proof of Theorem \ref{mt}: the case $d = n$}\label{opencase}In this subsection, we consider the case $d = n$.
It is not difficult to see (cf.~Corollary \ref{trivcorSC} below) that the function
 $SC_x^{\Lambda} (X)$ is equal to the characteristic function of
$C_x^{\Lambda} (X)$ almost everywhere.
Hence it is enough to prove that
$$
\Theta_d (X) (x)
=
\Theta_d (C_x^{\Lambda} (X)) (0)
$$
for $\Lambda$ small enough.

The proof we shall give is quite analogous to the one of Proposition 2.1 in \cite{KR}.
By Corollary \ref{localconic}, there exists a definable function
$\alpha : \PP^{n - 1} (K)
\rightarrow \NN$
and a subgroup $\Lambda$  in $\cD$,
such that for every $\ell$ in
$\PP^{n - 1} (K)$,
$(\pi_x^X)^{-1} (\ell) \cap B (x, \alpha (\ell))$ is a
local $\Lambda$-cone with origin $x$  in $(\pi_x^X)^{-1} (\ell)$.

For every $n \geq 0$, we consider the
$\Lambda$-cone $C_n (\Lambda)$ with origin $x$ generated by
$X \cap B (x, n)$.
Note that
$$C_x^{\Lambda} (X)
=
\cap_n \overline{C_n (\Lambda)},$$
hence, if we set
$$W :=
\cap_n C_n (\Lambda),$$
we have $W \subset C_x^{\Lambda} (X)$.
In particular,
$\Theta_d (W) (x) \leq
\Theta_d (C_x^{\Lambda} (X)) (0)$.
By Proposition \ref{convmon},
we know that
$\Theta_d (W) (x) =
\lim_n \Theta_d  (C_n (\Lambda)) (x)$
and
$\Theta_d (C_x^{\Lambda} (X)) (0) =
\lim_n \Theta_d  (\overline{C_n (\Lambda)}) (x)$.
By Proposition \ref{bord}, we deduce that
$$\Theta_d (W) (x) =
\Theta_d (C_x^{\Lambda} (X)) (0).$$

Since
we have
$$
\Theta_d (X) (x)
=
\Theta_d (X \cap B (x, n) ) (x)
\leq
\Theta_d (C_n (\Lambda)) (x),$$
we deduce that
$$
\Theta_d (X) (x) \leq \Theta_d (W) (x).
$$
To prove the reverse inequality, let us set consider for $n \geq 0$ the definable subset
$W_n$ of all points $w$ in $W$ such that $\alpha (\pi_x (w)) \leq n$.
By definition
$W_n \cap B (x, n) \subset X \cap B (x, n)$,
hence
$$\Theta_d (W_n) (x) =
\Theta_d (W_n \cap B (x, n)) (x)  \leq \Theta_d (X \cap B (x, n)) (x)
= \Theta_d (X) (x).
$$
Since, by Proposition \ref{convmon} again,
$\lim_n \Theta_d (W_n) (x) = \Theta_d (W) (x)$, we get
$\Theta_d (W) (x) \leq \Theta_d (X) (x)$, as required.

\subsection{Graphs}\label{gr}
The main technical result in this subsection is Proposition \ref{3.6}, which will be used
in subsection \ref{ept} to conclude the proof of Theorem \ref{mt} and in
subsection \ref{edt} to prove the existence of  distinguished
tangent $\Lambda$-cones.

Fix two integers $0<d \leq m$. Let $U$ be an open definable subset
of $K^d$ and let $\varphi$ be a definable  mapping $U \rightarrow
K^{m - d}$. The graph  $\Gamma = \Gamma (\varphi)$ of
$\varphi$ is a definable subset of $K^m$. Fix a point $u$ in the
closure $\overline U $ of $U$ and $\Lambda$ adapted to $(U, u)$. We
assume that $\displaystyle\lim_{x \to u}\varphi (x) = v$, by Corollary
\ref{cor1.8KR}. We set $w := (u, v)$.
The projection to the first $d$ coordinates $K^m \rightarrow K^d$
induces a function
$$
\vartheta :
\cD (\Gamma, w, \Lambda) \longrightarrow
\cD (U, u, \Lambda).
$$
Note that $\vartheta$ is an isomorphism
with inverse given by
\begin{equation}\label{theta-1}
\vartheta^{- 1}: (z, \lambda) \longmapsto (z, \lambda^{-1} (\varphi
(u + \lambda z) - v), \lambda).
\end{equation}

By Corollary \ref{localconic}, there exists a definable function
$\alpha : \PP^{d - 1} (K) \rightarrow \NN\cup\{\infty\}$ such that for every
$\ell$ in $\PP^{d - 1} (K)$, $(\pi_u^U)^{-1} (\ell) \cap B (u,
\alpha (\ell))$ is a local $\Lambda$-cone with origin $u$  in
$(\pi_u^U)^{-1} (\ell)$, with the convention that $\alpha(\ell)=\infty$
if and only if  $\ell$ is such that $(\pi_u^U)^{-1} (\ell) \cap U=\emptyset$. Note that being definable, the function
$\alpha$ is continuous on a dense definable open subset $\Omega_0$
in $\PP^{d - 1} (K)$, where dense means that the complement of
$\Omega_0$ has strictly smaller dimension. Let $\Omega_1$ be the
definable subset of $\Omega_0$ consisting of the $\ell$ such that
for all neighborhoods $ \cV$ of $u$ in $K^d$,
the sets $(\pi_u^U)^{-1} (\ell)\cap
\cV$ are nonempty.
\begin{lem}\label{nonempty}
Suppose that $C_u^{\Lambda} (U)$ is of maximal dimension. Then
$\Omega_1$ contains a nonempty open subset of $\PP^{d - 1} (K)$.
\end{lem}
\begin{proof}

Let $\Omega_1^c$ be the complement of $\Omega_1$ in $\PP^{d - 1}
(K)$. Clearly $\Omega_1^c$ is definable. It is enough to derive a
contradiction out of the assumption that $\Omega_1^c$ is dense.
Suppose thus that $\Omega_1^c$ is dense in $\PP^{d - 1} (K)$. By the
definability and density of $\Omega_1^c$ and of $\Omega_0$ in
$\PP^{d - 1} (K)$, it follows that $\Omega_1^c\cap \Omega_0$ is
dense.  Take $\ell$ in $\Omega_1^c\cap \Omega_0$. By the definition
of the tangent cone, one has that $(\pi_u^U)^{-1} (\ell) \cap
C_u^{\Lambda} (U)=\emptyset$. Since $(\pi_u^U)^{-1} ( \Omega_1^c\cap
\Omega_0 )$ is dense and definable in $K^d$, it follows that
$C_u^{\Lambda} (U)$ is contained in a definable subset of dimension
$<d$, a contradiction with $C_u^{\Lambda} (U)$ being of maximal
dimension, that is, of dimension $d$.
\end{proof}

For any definable subset $O$ of $\PP^{d - 1} (K)$, consider the
definable subset
\begin{equation}\label{CO}
C_u^{\Lambda, O} (U) := (\pi_u^U)^{- 1} (O)
\cap C_u^{\Lambda} (U)
\end{equation}
of $C_u^{\Lambda} (U)$.

\begin{lem}\label{nonempty2}
Suppose that $C_u^{\Lambda} (U)$ is of maximal dimension. Let
$O$ be dense open in $\Omega_1$. Then the
set $C_u^{\Lambda, O} (U)$ also has maximal dimension and is dense
open in $C_u^{\Lambda} (U)$.
\end{lem}
\begin{proof}
The fact that
 $C_u^{\Lambda, O} (U)$ is open in $C_u^{\Lambda} (U)$ follows
from general topology. We only have to prove that $C_u^{\Lambda, O}
(U)$ is dense in $C_u^{\Lambda} (U)$. As we have noticed in the proof
of Lemma \ref{nonempty},  for every $\ell$ in $\Omega_1^c\cap \Omega_0$,
\[
(\pi_u^U)^{-1} (\ell) \cap
C_u^{\Lambda} (U)=\emptyset.
\]
Since moreover the sets $(\pi_u^U)^{-1} (\Omega_1 \setminus O)$ and
$(\pi_u^U)^{-1} (\Omega_1^c \setminus \Omega_0)$ have dimension
$<d$, the lemma follows.\end{proof}


The next lemma ensures in particular that there definable sets $\Omega=O$ as in Lemma \ref{nonempty2} such
that moreover for all $z$ in $C_u^{\Lambda, \Omega}(U)$ and all small enough
$\lambda$ in $\Lambda$ one has that $u + \lambda z$ lies in $U$.

\begin{lem}\label{smallU}
Suppose that $C_u^{\Lambda} (U)$ is of maximal dimension. Then there
is a dense open definable subset $\Omega$ of $\Omega_1$ such that
for all $z$ in $C_u^{\Lambda, \Omega}
(U)$ of direction $\ell$ and for all small enough $\lambda$ in $\Lambda$
one has $u+\lambda z\in \overline{(\pi_u^U)^{-1}(\ell) }$. Here, small enough can be taken to mean that $\ord (\lambda z)\geq \alpha(\ell)$, where $\alpha$ is as in the beginning of section \ref{gr}.
\end{lem}
\begin{proof}
We assume $u=0$ for
simplicity. For any $\ell\in \Omega$,
any $z\in \ell \cap  C_u^{\Lambda}(U)$ and any $\lambda\in \Lambda$,
one has $\lambda z\in  C_u^{\Lambda}(U)$. Hence what remains to be proved is
a consequence of the inclusion $\subset$ in the equality of  the
following claim.
\end{proof}

\begin{claim}
For almost all $\ell$ in $\Omega_1$ and with $u=0$,
 one has the following equality of local $\Lambda$-cones
$$
C_u^{\Lambda} (U)\cap \ell \cap  B (0, \alpha (\ell))   =
(\overline{\pi_u^{U})^{-1} (\ell)} \cap B (0, \alpha (\ell)).
$$
\end{claim}
\begin{proof}[Proof of the claim]
Since $\ell\in \Omega_1$,
$ (\pi_u^U)^{-1}(\ell)\cap B(0,\alpha(\ell))$ is a local $\Lambda$-cone
with origin $0$, and by Remark
\ref{remarque sur les cones locaux}, we have
$$
\overline{(\pi_u^U)^{-1}(\ell)} \cap B(0,\alpha(\ell))
=C_u^\Lambda( (\pi_u^U)^{-1}(\ell)) \cap B(0,\alpha(\ell))
$$
$$ \subset C_u^{\Lambda} (U)\cap \ell \cap  B (0, \alpha (\ell)) .$$
\noindent
The inclusion $\supset$ is thus clear for all $\ell\in \Omega$.
We prove the inclusion $\subset$ in the claim, for almost all $\ell$.
 This follows from cell decomposition. Let $X\subset
(K^d\setminus\{0\})\times K$ be the definable set
$$
\{(x,t)\in (K^d\setminus\{0\})\times K\mid u+ t\cdot x\in U\}, $$
 parametrizing all $\ell\cap U$ for all lines $\ell$ through $u$.
 Then $X$ is a finite union of cells by Theorem
 \ref{thm:CellDecomp}. For each $x\in K^d\setminus\{0\}$ write $X_x$
 for the fiber above $x$ under the projection $X\to K^d$. For each
 $x$, either $0$ lies in the interior of $X_x$, either $0$ lies in
 the boundary  $\partial X_x$ of $X_x$ or $0$ lies outside the closure of $X_x$, where $\partial X_x$
 is the closure of $X$ minus the interior of $X$. In the case that $0$ lies in the interior of
 $X_x$, one has that $(\pi_u^U)^{-1} (\ell) \cap B (u, \alpha
 (\ell))= B (u, \alpha
(\ell))$ hence the inclusion $\subset$ is evident.

 The inclusion $\subset$ holds, up to a  set of direction $\ell\in
\PP^{d-1}(K)$ of dimension $<d-1$, for
 all those $x$ such that $0$ lies in $\partial X_x$ by the almost
 everywhere continuity of the functions in $x$ appearing in the
 descriptions of the cells having $0$ in their boundary. The case
 that $0$ lies outside the closure of $X_x$ needs not to be
 considered since we suppose $\ell\in\Omega_1$.
\end{proof}

\begin{cor}\label{distinguished cone for open sets}
Let $d>0$, let $U$ be a definable nonempty open subset of $K^d$ and
let $\Lambda$  given by Corollary \ref{localconic}.
Then for all $u\in \overline U$,
with $C_u^\Lambda(U)$ of maximal dimension $d$,
$C_u^{\Lambda}(U)$ is a distinguished $\Lambda$-tangent cone at $u$ for $U$,
that is, for all $\Lambda'\in \cD$,
$\Lambda'\subset \Lambda$ implies $C_u^{\Lambda'}(U)=C_u^{\Lambda}(U)$.
\end{cor}

\begin{proof}
As usual we assume $u=0$.
Let $\Lambda$ as given by Corollary \ref{localconic} and
$\Lambda'\in \cD$, $\Lambda'\subset \Lambda$.
We show that $C_u^{\Lambda}(U)\subset C_u^{\Lambda'}(U)$.
Let $z\in C_u^{\Lambda}(U)$, denote by $\ell$
its direction and assume that $\ell \in \Omega$, with
$\Omega\subset \Omega_1$ as in Lemma \ref{smallU}.
By Lemma \ref{smallU} we then have
$z\in C_u^\Lambda(\overline{U\cap \ell})= C_u^\Lambda(U\cap \ell )$.
 But since
$U\cap \ell\cap B(0,\alpha(\ell))$ is a local $\Lambda$-cone with origin
$u$, by Remark
\ref{remarque sur les cones locaux}, we get
$$z \in C_u^{\Lambda}(U\cap \ell)
=C_u^{\Lambda'}(U\cap \ell)\subset
C_u^{\Lambda'}(U).$$
Now since we showed
$C_u^{\Lambda,\Omega}(U)=C_u^{\Lambda',\Omega}(U)$, we have
$\overline{C_u^{\Lambda,\Omega}(U)}=\overline{C_u^{\Lambda',\Omega}(U)}$.
But by Lemma \ref{nonempty2} we obtain $\overline{C_u^{\Lambda,\Omega}(U)}
=C_u^{\Lambda}(U)$. We finally remark that one also has
$\overline{C_u^{\Lambda',\Omega}(U)}
=C_u^{\Lambda'}(U)$, with the same proof as in Lemmas \ref{nonempty}
and \ref{nonempty2}, since any adapted (to $U$) $\Lambda'$-cone may be
chosen in those lemmas.
\end{proof}

 Proposition \ref{3.6} below
can be seen as an analogue of Proposition 3.6 of \cite{KR} and has for
consequence the existence of distinguished $\Lambda$-tangent cones for
general definable sets and the $p$-adic analogue of Thie's formula.
As usual, the main point in the $p$-adic case is to overcome the lack of
connectedness and
deal with all its negative consequences such as the lack of
strong enough
mean value theorems
 and so forth. To go through these difficulties we essentially
use the following result, the main result of \cite{CCL}, which is the
$p$-adic analogue of the existence of the so-called $L$-decompositions
of real subanalytic sets, obtained in \cite{KurWExp}, and which will be used in the proof of Proposition \ref{3.6} (c).

\begin{theorem}
\label{Lipschitz decomposition}
Let $\varepsilon>0$ and let $\varphi:X\subset K^n\to K^m$ be a
locally $\varepsilon$-Lipschitz definable mapping. Then there exist
$C>0$ and a finite definable partition of $X$ into parts $X_1, \cdots, X_k$,
such that the restriction of  $\varphi$ to each $X_i$ is globally $C$-Lipschitz.
\end{theorem}
We will also use Theorem \ref{wf} in the proof of
Proposition \ref{3.6}
in the same way that Lemma 3.7 of \cite{KR} is used in the proof of
Proposition 3.6 of \cite{KR}. Where Theorem \ref{wf}
gives the existence of $(w_f)$-regular (and consequently $(a_f)$, $(b)$, and
$(w)$-regular) stratifications for a function
$f$ in the definable $p$-adic setting,
we will only use the genericity
of the condition $(a_f)$ in the $p$-adic definable case to prove Proposition \ref{3.6}.

\begin{prop}\label{3.6}
Let $\varepsilon$ be a positive real number with $\varepsilon\leq 1$.
Let $U$ be an open definable subset of $K^d$ and let $\varphi$ be a
definable  mapping $U \rightarrow K^{m - d}$. Fix a point $u$ in
$\overline U$, a subgroup $\Lambda$ adapted to $(U, u)$, choose $\Omega$
sufficiently small
and as
in Lemma \ref{smallU},
and let $C_u^{\Lambda, \Omega} (U)$ be as in
\textup{(}\ref{CO}\textup{)}. Assume that $\varphi$ is $\varepsilon$-analytic
and that  $\displaystyle\lim_{x \to u}\varphi (x) = v$
by Corollary \ref{cor1.8KR}.

Suppose that $C_u^{\Lambda} (U)$ has maximal dimension. Then,
possibly after partitioning $U$ into finitely many open subsets, replacing
$U$ successively by each one of these smaller open subsets, in such a way
that $\varphi$ is globally $C$-Lipschitz on $U$ by Theorem \ref{Lipschitz
decomposition} and neglecting
those $U$ such that $C_u^{\Lambda} (U)$ has lower dimension, the
following hold
\begin{enumerate}
\item[(a)]
For $z$ in $C_u^{\Lambda, \Omega} (U)$ such that $u+\lambda z\in U$,
for all small enough $\lambda\in \Lambda$ \textup{(}see Lemma \ref{smallU}\textup{)},
the limit
$$
\psi (z) := \lim_{\sur{\lambda \to  0}{\lambda \in \Lambda}}
\lambda^{-1} (\varphi (u + \lambda z) - v)
$$
exists in $K^{m-d}$, yielding a definable function
$\psi:C_u^{\Lambda, \Omega} (U)\to K^{m-d}$.

\item[(b)]The function $\psi$ is locally $\varepsilon$-Lipschitz.

\item[(c)] The graph of $\psi$ is dense in
$C_w^{\Lambda}(\Gamma(\varphi))$.

%
%
%
\end{enumerate}
\end{prop}
\begin{proof}
We first prove (a). Choose
$z$ in $C_u^{\Lambda, \Omega} (U)$ such that $u+\lambda z\in U$,
for all small enough $\lambda\in \Lambda$. We can
evaluate $\varphi$ at $u+\lambda z$
for small enough $\lambda$ in $\Lambda$. After partitioning $U$ into
finitely many open subsets and successively replacing $U$ by each one of
these smaller open subsets, Lemma \ref{terms} implies that
when $\lambda \to 0$,  $\lambda \in \Lambda$, either
$ \lambda^{-1}
(\varphi (u + \lambda z) - v)$
has a limit $\psi (z)$ or its norm goes to $\infty$.
 Applying the Curve Selection Lemma \ref{curve} to the point
$(u,v)$ and the set
$$\{(u+\lambda z,\varphi (u + \lambda z) )\in
K^m\mid \lambda\in \Lambda\},$$ it follows from Lemma \ref{lemWhitney} and
$\varepsilon$-analyticity that the limit $\psi (z)$ exists.

Now assume that $z\in C_u^{\Lambda, \Omega} (U)$ is such that
$u+\lambda z\in \overline{U\cap \ell}$, where $\ell$ is the (direction
of the) line going through $u$ and $z$.
By Lemma \ref{smallU}, for all $\varepsilon>0 $ there exist $z',z''\in
C_u^{\Lambda, \Omega} (U)$ such that $u+\lambda z'\in U$,
$u+\lambda z''\in U$
for all $\lambda\in \Lambda$ and $\vert z-z'\vert \le \varepsilon$ and
$\vert z-z''\vert \le \varepsilon$. Then we have
$$\vert \lambda^{-1} [ (\varphi(u+\lambda z')-v)-
(\varphi(u+\lambda z'')-v) ]\vert \le C\vert \lambda^{-1}
 \lambda(z'-z'')\vert \le C\cdot \varepsilon.$$
 This shows that one can define $\psi$ on $C_u^{\Lambda,\Omega}$ by
\begin{equation}\label{psizz'}
\psi(z)=\displaystyle \lim_{z'\to z}\lim_{\lambda\to 0}\lambda^{-1}
(\varphi(u+\lambda z')-v).
\end{equation}

Let us now prove  (b).
We first notice that, after
suitable finite partition of $U$ and neglecting those $U$ such that
$C_u^{\Lambda} (U)$ has lower dimension, we may suppose that the
function $\psi$ is analytic on $C_u^{\Lambda, \Omega} (U)$.
To prove that $\psi$ is $\varepsilon$-analytic on
$C_u^{\Lambda, \Omega} (U)$,  we show that the tangent space
$T_x\Gamma(\psi)$ at a point $x$ of the graph $\Gamma(\psi)$ of $\psi$,
for $x$ in a dense set of $\Gamma(\psi)$, is contained in
$C_\varepsilon=\{(a,b)\in K^d\times K^{m-d};
\vert b\vert \le \varepsilon \vert a\vert\}$. Since,
by (\ref{psizz'}),
$\Gamma(\psi)\subset C_w^{\Lambda}(\Gamma(\varphi))$ and since
$\dim(\Gamma(\psi))=\dim(C_w^{\Lambda}(\Gamma(\varphi)))$,
 it is enough to
prove that at a generic point $x$ of $C_w^{\Lambda}(\Gamma(\varphi))$
one has $T_xC_w^{\Lambda}(\Gamma(\varphi))\subset
C_\varepsilon$. For this we consider the deformation
$h:\overline{{\cD}(\Gamma(\varphi),w, \Lambda)}\to K$ to
$C_w^{\Lambda}(\Gamma(\varphi))$
defined in section $3.5$. The fiber $h^{-1}(0)$ is identified with
$C_w^\Lambda(\Gamma(\varphi))$ and for $\lambda\in \Lambda$
$$h^{-1}(\lambda)=\{(z,\lambda)\in K^m\times \Lambda;
w+\lambda z\in \Gamma(\varphi)\}$$ is identified
with
$$\{z\in K^m; w+\lambda z\in \Gamma(\varphi)\}.$$
Since $\varphi$ is $\varepsilon$-analytic, for any
$\lambda\in \Lambda$ and any $y\in h^{-1}(\lambda)$,
one has $T_y h^{-1}(\lambda)\subset C_\varepsilon$.
Let us show at $x$ a generic point of $C_w^\Lambda(\Gamma(\varphi))$,
$T_xC_w^\Lambda(\Gamma(\varphi))$ is a limit of tangents
$T_{y_n} h^{-1}(\lambda_n)$. But this is exactly the genericity in $h^{-1}(0)$
of the condition $(a_h)$, which is given by Theorem \ref{wf}.
\\

We now prove (c). Let
$z\in C_w^\Lambda(\Gamma(\varphi))$
 and
$(\lambda_n)_{n\in \NN}\in \Lambda,
(w_n)_{n\in \NN}\in \Gamma(\varphi)$ be
two sequences such that $w_n\to w$ and $\lambda_n(w_n-w)\to z$.
Denoting by $\pi$ the projection from $\Gamma$ to $U$ and
$u_n=\pi(w_n)$, the sequence $(u_n)_{n\in \NN}$ of points of $U$ going
to $u$ is such that $\displaystyle\lim_{n\to \infty}\lambda_n(u_n-u)=
\pi(z):=a\in
C^\Lambda_u(U)$. Now
fix $\varepsilon>0$ and $a'\in C_u^{\Lambda,\Omega}(U)$ with
$\vert a-a' \vert\le \varepsilon $. Then $u+\lambda a'\in U $ for all small
enough $\lambda\in \Lambda$ by Lemma \ref{smallU}.
 Then we may suppose, by invoking Theorem \ref{Lipschitz decomposition}, that
$$\vert \lambda_n(\varphi(u_n)-v)-\lambda_n(
\varphi(\lambda_n^{-1}a'+u)-v) \vert \le C
\vert \lambda_n (u_n-u)-a')\vert. $$
This gives
$$ \lim_{n \to \infty}
\vert \lambda_n(w_n-w)-
\lambda_n(\lambda_n^{-1}a',\varphi(\lambda_n^{-1}a'+u)) \vert
\le \max(C,1)\cdot \varepsilon, $$
and, finally
$$
\vert z-(a',\psi(a')) \vert \le \max(C,1)\cdot \varepsilon,
$$
showing that the graph of $\psi $ is dense in
$C_w^\Lambda(\Gamma(\varphi))$.
\end{proof}
\begin{cor}\label{trivcorSC}
Under the hypotheses
and with the notation  of Proposition \ref{3.6}, assume moreover that $\varepsilon < 1$. Write $z$ for variables running over $K^{d}$ and $y$ for variables running over $K^{m-d}$.
Then, for almost all $(z,y)\in C_w^{\Lambda} (\Gamma(\varphi))$, one has that  $SC_w^{\Lambda} (\Gamma(\varphi))(z,y) = 1 = SC_u^{\Lambda} (U)(z) $, and
$$
\Theta_d (SC_w^{\Lambda} (\Gamma(\varphi))) (0)
=
\Theta_d (C_w^{\Lambda} (\Gamma(\varphi))) (0) = \Theta_d (U) (u)=
\Theta_d (\Gamma(\varphi)) (w).
$$
\end{cor}
\begin{proof}

We first prove that, for almost all $z\in C_u^{\Lambda} (U)$, one has that  $1 = SC_u^{\Lambda} (U)(z) $. Let $\Omega$ and $\alpha$ be as in Lemma \ref{smallU}. By Lemma \ref{smallU}, for $z\in C_u^{\Lambda, \Omega}
(U)$, there exist an open ball $B$ contained in $C_u^{\Lambda, \Omega}
(U)$ and containing $z$, and a ball $B_1\subset K$ around $0$ such that
$$
\cD(U,u,\Lambda)\cap (B\times B_1 ) =    B\times  ( \Lambda\cap B_1 ).
$$
Hence we can calculate
$$
 SC_u^\Lambda (U)(z)
 =
 [K:\Lambda]\Theta_{d+1}(\cD (U,u,\Lambda))(z,0)=
[K:\Lambda]\Theta_{d+1}(B \times \Lambda )(z,0)=\Theta_{d}(B )(z)
$$
which equals $1$ since $z\in B$.

Next we prove that $SC_w^{\Lambda} (\Gamma(\varphi))(z,y)=1$ for almost all $(z,y)\in C_w^{\Lambda} (\Gamma(\varphi))$.
For this purpose, define
$$
\cD' := \{(z,y,\lambda)\in \cD(\Gamma(\varphi),w,\Lambda);\ |y| <  |z| \},
$$
and consider the natural projection
$$
p: \cD'  \to \cD (U,u,\Lambda): (z,y,\lambda)\mapsto (z,\lambda),
$$
which is in fact injective.
Write $U'$ for the image of $p$.
By Proposition \ref{prop1.7KR} and Lemma \ref{curve}, one finds for all $(z,y)\in K^m$ that
$$
\Theta_{d+1}( \cD'  )(z,y,0)   = \Theta_{d+1}( \cD(\Gamma(\varphi),w,\Lambda) )(z,y,0)
$$
and for almost all $z\in K^{d}$ that
$$
\Theta_{d+1}( U' )(z,0) = \Theta_{d+1}( \cD(U,u,\Lambda) )(z,0).
$$
Since for all $(z,y,\lambda)\in \cD' $ one has $|(z,y,\lambda)  | = |(z,\lambda) |$, by the bijectivity of $p:\cD'\to U'$,  and by definition of $\Theta_{d+1}$, one finds
$$
\Theta_{d+1}( \cD'  )(z,y,0) = \Theta_{d+1}( U' )(z,0).
$$
 This shows that $SC_w^{\Lambda} (\Gamma(\varphi))(z,y)=1$ for almost all $(z,y)\in C_w^{\Lambda} (\Gamma(\varphi))$. It also follows that
$$
\Theta_d (SC_w^{\Lambda} (\Gamma(\varphi))) (0)
=
\Theta_d (C_w^{\Lambda} (\Gamma(\varphi))) (0).
$$
We proceed with similar arguments to show the remaining equalities.
 Assume from now on until the end of the proof, for simplicity, that $w=0$.
By Proposition \ref{3.6} (c) one has
$\Theta_d (\Gamma (\psi)) (0) = \Theta_d (C_w^\Lambda(\Gamma ( \varphi)) (0)$.
By Propositions \ref{prop1.7KR} and \ref{3.6} and since
$\varepsilon < 1$, the map $z \mapsto (z, \psi (z))$ defined for $z$ in $C_u^{\Lambda, \Omega} (U)$ preserves the norm in the sense that $|z| =|(z,\psi(z) )|$ (recall that one uses the $\sup$-norm for tuples in an ultrametric setting). Hence, by the definition of $\Theta_d$, by Lemma \ref{nonempty2} and by section \ref{opencase}, one has  $\Theta_d (\Gamma (\psi)) (0) = \Theta_d (C_u^{\Lambda, \Omega} (U)) (0) = \Theta_d (C_u^{\Lambda} (U)) (0) = \Theta_d (U) (0)$.
Combining the obtained series of equalities yields $\Theta_d (SC_w^{\Lambda} (\Gamma(\varphi))) (0) = \Theta_d (U) (0)$.

Finally we prove that this value also equals $\Theta_d (\Gamma(\varphi)) (0)$, by again a similar argument.
Define $U'':=\{z\in U,\  |\varphi(z)| < |z| \}$. Then, by Lemma \ref{prop1.7KR} and its proof based on Lemma \ref{curve}, we find $C_u^{\Lambda} (U) = C_u^{\Lambda} (U'')$. Hence, $\Theta (C_u^{\Lambda} (U))(0) = \Theta_d (C_u^{\Lambda} (U'') )(0)$ which also equals $ \Theta_d (U''  )(0)  $ by section \ref{opencase}.
  Since on  $U''$  the map $z\mapsto (z,\varphi(z))$ preserves the norm in the sense that $|z| = |(z,\varphi(z))|$ we find by the definition of $\Theta_d$ that
 $  \Theta_d(U'')(u) = \Theta_d ( \Gamma(\varphi_{|U''}) ) (w) = \Theta_d ( \Gamma(\varphi) ) (w)$ which finishes the proof.
\end{proof}

\subsection{An alternative view on cones with multiplicities.}

Let $X\subset K^n$ be definable and of dimension $d$.
It follows from Proposition \ref{3.6} and its corollary \ref{trivcorSC}  that there is a finite definable partition of $X$ into parts $X_j$ which are graphs of $\varepsilon$-analytic Lipschitz functions on open subsets $U_j$, such that, for small enough $\Lambda$, $\Theta_d(X_j)(0)=\Theta_d(C^\Lambda_0(X_j))(0)$ for each $j$.
 It follows by additivity that
$$
\sum_j \Theta_d (C^\Lambda_0(X_j)) = \sum_j \Theta_d(X_j)=\Theta_d (X).$$
This common value can of course can be different from $\Theta_d(C_0^\Lambda(X))$ since $X_j$ et $X_k$
may have tangent cones which coincide on a part of dimension $d$ for different $j$, $k$, that is, there might be overlap in the union $C_0^\Lambda(X) = \cup_j C_0^\Lambda(X_j)$.
Let us decompose $C_0^\Lambda(X)$ into parts $C_k$, $k\geq 1$, with the property that a line $\ell\subset  C_k$ (through the origin) belongs to $C_0^\Lambda(X_j)$ for exactly $k$ different $j$. (Note that such decomposition is in general not unique.)
Let us then define the function $CM_0^{\Lambda} (X)$ on $C_0^{\Lambda} (X)$, up to definable subsets of $C_0^{\Lambda} (X)$ of dimension $<d$, by
 $$
 CM_0^{\Lambda} (X)=\sum_k k\cdot \11_{C_k}.
 $$
Any other such decomposition of $X$ into parts $X_j$ will yield the same function $CM_0^{\Lambda} (X)$ up to a definable subset of $C_0^{\Lambda} (X)$ of dimension $<d$, as can be seen by taking common refinements and by general dimension theory of definable sets.
Clearly $\Theta_d(CM_0^{\Lambda} (X))=
\sum_j \Theta_d(X_j)=\Theta_d (X)$.
Moreover, by additivity of $SC$ and by Corollary \ref{trivcorSC}, for all $z\in C_0^{\Lambda} (X)$ up to a definable set of dimension $<d$
$$
SC_0^{\Lambda} (X)(z) = CM_0^{\Lambda} (X)(z).
$$
In particular it follows that $SC_0^{\Lambda} (X)(z)$ is a nonnegative integer for all $z\in C_0^{\Lambda} (X)$ up to a definable set of dimension $<d$.

\subsection{End of proof of
Theorem \ref{mt}}\label{ept}We consider a definable subset $X$ of dimension $d$
in $K^n$ and a point $x$ of $K^n$. We may assume $x$ lies in the closure of $X$.
Let us fix $0 < \varepsilon < 1$.
By Proposition \ref{1.4} there exists a decomposition
$$X =  \bigcup_{1 \leq i \leq N (\varepsilon)} \gamma_i (\Gamma_i (\varepsilon)) \cup Y$$
with
$Y$ a definable subset of $X$ of dimension $< d$,
definable open subsets
$U_i (\varepsilon)$ of $K^d$, for
$1 \leq i \leq N (\varepsilon)$,
definable analytic functions
$\varphi_i (\varepsilon) : U_i (\varepsilon) \rightarrow K^{m - d}$ whose graphs
$\Gamma_i (\varepsilon)$
are all $\varepsilon$-analytic,
and elements $\gamma_1$, \dots, $\gamma_{N (\varepsilon)}$ in
$\GL_m (R)$ such that the sets $\gamma_i ( \Gamma_i (\varepsilon))$ are all disjoint and contained in $X$.
We denote by  $u_i$  the image of
$\gamma_i^{-1} (x)$ under the projection to $K^d$
and we fix  $\Lambda$ adapted to
$(X,x)$ and to
$(U_i (\varepsilon), u_i)$ for every
$1 \leq i \leq N (\varepsilon)$.
By linearity and since the $\gamma_i$'s are isometries,
we have then
$$\Theta_d (X) (x) = \sum_{1 \leq i \leq N (\varepsilon)}
\Theta_d (\Gamma_i (\varepsilon)) (\gamma_i^{-1} (x))$$
and
$$\Theta_d (SC_x^{\Lambda}(X)) (0) = \sum_{1 \leq i \leq N (\varepsilon)}
\Theta_d (SC_{\gamma_i^{-1} (x)}^{\Lambda}(\Gamma_i (\varepsilon))) (0),$$
and the result follows from Corollary \ref{trivcorSC}. \qed

\subsection{Existence of distinguished
tangent $\Lambda$-cones}\label{edt}
We deduce from Corollary \ref{distinguished cone for open sets} and
Proposition \ref{3.6} the existence of distinguished
tangent $\Lambda$-cones.

\begin{theorem}\label{distinguished cone}
Let $X$ be a definable subset of $K^n$. Then there exists $\Lambda\in
\cal D$ such that for any $x\in \overline X$, $C_x^\Lambda(X)$ is a
distinguished $\Lambda$-cone, that is to say
$\Lambda'\subset \Lambda$
implies $C_x^{\Lambda'}(X)=C_x^\Lambda(X)$.
\end{theorem}

\begin{proof}

We will work by induction on the dimension $d$ of $X$, where for $d=0$ the statement is trivial.
We may work up to a finite partition of $X$ into definable pieces $X_k$ with distinguished $\Lambda_k$-cones $C_x^{\Lambda_k}(X_k)$ for all $x$ and for some $\Lambda_k$, since one can put $\Lambda := \cap_k \Lambda_k$ and then $C_x^\Lambda(X)=\cup_k C_x^\Lambda(X_k)$ implies that $C_x^\Lambda(X)$ is a distinguished $\Lambda$-cone for all $x$.
Up to a finite partition using Proposition \ref{1.4} and Theorem \ref{Lipschitz decomposition}, we may suppose that $X$ is the graph of some definable $C$-Lipschitz and
$\varepsilon$-analytic map $\varphi:U\to K^{n-d}$, where $U$ is a definable open subset of
$K^{d}$ and $d$ is the dimension of $X$.

Fix $x\in \overline X$ and write $u\in \overline U$ for the projection of $x$ in $K^d$. We will construct a distinguished $\Lambda$ for this fixed $x$, with the extra property that in the construction one could as well take $x$ as a parameter running over $K^n$ and consider the analogue of the set-up in families parameterized by $x$, and then only finitely many $\Lambda$ will come up. Taking the intersection of these finitely many $\Lambda$ as above then finishes the proof.

First suppose that $\varphi$ falls under the conditions of Proposition \ref{3.6}, that is,
 $C_{u}^\Lambda(U)$ has maximal dimension $d$ for some $\Lambda\in \cal D $ which is adapted to $U$.
  We know from Corollary \ref{distinguished cone for open sets} that
$\Lambda$ is distinguished for $U$, meaning that for $\Lambda'\subset \Lambda$
one has
\begin{equation}\label{lambdal'}
C_{u}^{\Lambda'}(U) =C_{u}^{\Lambda}(U).
\end{equation}
Fix $\Lambda'\subset \Lambda$ and consider
$$
\psi: C_{u}^{\Lambda,\Omega}(U)\to
K^{n-d}
$$
and
$$
\psi': C_{u}^{\Lambda',\Omega'}(U)\to
K^{n-d}$$
(the notation being coherent with Proposition \ref{3.6}).
We may suppose that $\Omega = \Omega'$. But then  $C_{u}^{\Lambda,\Omega}(U) = C_{u}^{\Lambda',\Omega'}(U)$  by Equation (\ref{lambdal'}), and, for any $z$ in this set, we have $\psi(z)=\psi'(z)$ by Proposition
\ref{3.6} (a). Hence, $\psi$ and $\psi'$ are the same function. Taking the closures of the graph of this function, Proposition
\ref{3.6} (c) now yields that $C_x^{\Lambda'}(X) = C_x^{\Lambda}(X)$ and we are done in this case.

Let us finally consider the case that $C^\Lambda_x(X)$ has dimension $<d$ for some $\Lambda$ (which happens if and only if $C^\Lambda_u(U)$ has dimension $<d$). We will construct a definable $Y\subset X$ such that $\dim(Y)=\dim(C^\Lambda_x(X))$ and
$C_x^\Lambda(Y)=C_x^\Lambda(X)$. Then we can replace $X$ by $Y$ and we are done by induction on the dimension.

Let  $h:\overline{\cal D(X,x,\Lambda)}
\to K$ be the deformation to $C_x^\Lambda(X)$.
We assume $x=0$ in what follows, though
we keep the notation $x$.
Let $\cal L(C^\Lambda_x(X))$ be
$$
C^\Lambda_x(X)\cap \displaystyle\bigcup_{i=0}^{e-1}S(0,i),
$$
where $e=[K^\times:\Lambda]$. We call  $\cal L(C^\Lambda_x(X))$ the $\Lambda$-link of $C^\Lambda_x(X)$.
Note that the $\Lambda$-cone generated by $\cal L(C^\Lambda_x(X))$ equals $C^\Lambda_x(X)$. Let $\tilde{\cal L}(C^\Lambda_x(X))$ be
${\cal L}(C^\Lambda_x(X))\times (B(0,n)\cap \Lambda)$ for some ball $B(0,n)$ around $0$.
Since there are definable choice functions, there is a map
$$
d: \tilde{\cal L}(C^\Lambda_x(X))\to \overline{\cal D(X,x,\Lambda)}
$$
with $ d(z,\lambda)\in h^{-1}(\lambda)$ for all $\lambda$ and $\lim_{\lambda\to 0 }
d(z,\lambda)=z$ for all $z$. Since we may and do suppose that $z\not = d(z,\lambda)$, the image of $d$ is
of dimension $\dim (\tilde{\cal L}(C^\Lambda_x(X)) ) = \dim(\cal L(C^\Lambda_x(X)))+1=\dim(C^\Lambda_x(X))+1$. We send
$d( \tilde{\cal L}(C^\Lambda_x(X)))$ into  $X$ by
$r(z,\lambda)=\lambda\cdot z$ and we set
$Y=r(d( \tilde{\cal L}(C^\Lambda_x(X))))$. Then $Y$ is a definable
subset of $X$ of dimension $\dim(C^\Lambda_x(X))$ and by construction
$C^\Lambda_x(Y)=C^\Lambda_x(X)$.
\end{proof}

\section{A local Crofton formula}

\subsection{Local direct image}Let $p : X \rightarrow Y$
be a definable function between two definable sets of the same
dimension $d$. If $\varphi$ is a function in $\cC (X)$ and $y$ is in
$Y$ we set $p_! (\varphi) (y) = \sum_{x \in p^{- 1}(y)} \varphi
(x)$ if $p^{- 1}(y)$ is finite and $p_! (\varphi) (y) = 0$ if it
is infinite. The function $p_! (\varphi)$ lies in $\cC (Y)$,
since the cardinality of $p^{- 1}(y)$ takes only finitely many values when $y$ runs over $Y$.

If $X$ is a definable subset of  $K^n$ and  $x$ is a point
of $K^n$, we define the algebra
$\cC (X)_x$ of germs of  constructible functions in $\cC (X)$ at $x$ to be the quotient
of $\cC (X)$ by the equivalence relation
$\varphi \sim \varphi'$ if $ \11_{B (x , n)} \varphi =
\11_{B (x , n)} \varphi'$ for $n$ large enough. That definition is only relevant when
$x$ is in the closure of $X$.
Also, if $\varphi$ is in
$\cC (X)_x$ is the germ of a locally
bounded function $\psi$, $\Theta_d (\varphi)  :=\Theta_d (\psi)
(x)$ does not depend on the representative $\psi$.

Let $p : K^m \rightarrow K^{d}$ be a linear projection and let
$X$ and $Y$ be respectively definable subsets of $K^m$ and $K^{d}$
such that $p (X) \subset Y$.
Fix
$x$ in $K^m$.
When the  condition (\ref{sparks}) is satisfied
\begin{equation}\label{sparks}\tag{$\ast$}
\text{there exists}\, \,  n \geq 0 \, \, \text{such that} \, \,
p^{-1} (p (x)) \cap \overline{X \cap B (x, n)}= \{x\},
\end{equation}
then, for every function $\varphi$ in $\cC (X)$, the class
of $p_! (\varphi \11_{B (x, n)})$ in $\cC (Y)_{p (x)}$
does not depend on $n$ for $n$ large enough.
We denote it by $p_{!, x} (\varphi)$. We also denote by
$p_{!, x}$ the corresponding morphism
$\cC (X)_x \rightarrow \cC (Y)_{p (x)}$.

\subsection{The local Crofton formula for the local density}
For $x$ a point in $K^n$ we consider $K^n$ as a vector space with
origin $x$ and
for $0 \leq d \leq n$, we denote by $G (n, n-d)$ the corresponding
Grassmannian of
$(n-d)$-dimensional  vector subspaces of
$K^n$. It is a compact $K$-analytic variety, endowed with a unique
measure $\mu_{n, d}$ invariant under $\GL_n (R)$
and such that
$\mu_{n, d} (G (n, n-d)) = 1$.

For any $V$ in $G (n, n-d)$,
we denote by $p_V : K^n \rightarrow K^n/V$, the canonical
projection,
where $K^n/V$ is identified with the $K$-vector space $K^d$.
This identification enables the computation of the local
density
of germs
in $K^n/V$.

Let $X$ be a definable subset of $K^n$ of dimension $d$ and let $x$ be a
point of $K^n$. By general dimension theory for definable sets
there exists a dense definable open subset $\Omega(=\Omega_X)$ of
$G(n, n-d)$ such that for every
$V$ in $\Omega$ the projection
$p_V$ satisfies the condition \textup{(}\ref{sparks}\textup{)}
with respect to $(X, x)$.

The following statement is the $p$-adic analogue of
the so-called {\sl local Crofton formula} proved in \cite{comte}
for real subanalytic sets
and more generally in \cite{lk}, again in the real subanalytic
setting, in its
multidimensional
 version.
\begin{theorem}\label{lcc}
Let $X$ be a definable subset of $K^n$ of dimension $d$ and let
$x$ be a point of $K^n$. Let $\varphi$ in $\cC (X)_x$ be the germ
of a locally bounded function.
Then
$$
\Theta_d (\varphi) (x)
= \int_{V\in \Omega \subset G (n, n-d)}
\Theta_d (p_{V !, x} (\varphi)) \,
d \mu_{n, d}(V).
$$
\end{theorem}

We may assume that $X=\overline X$ by Proposition \ref{bord} and
that $x=0$ and $0\in X$, for if $0\not\in X$,
$\Theta_d(X)(0) = \Theta_d (p_{V !, x} (\varphi))=0$ (for generic $V$)
 and the statement of Theorem \ref{lcc} is then true.

In order
to emphasize
 the geometric-measure part of \ref{lcc}
we start with the following lemma,
which is Theorem \ref{lcc} for $X$ a
definable $\Lambda$-cone of $K^n$ of dimension $d$ contained in some
$d$-dimensional vector space of $K^n$ and $\varphi=\11_X$.
\begin{lem}\label{lccconeplan}
 Let $\Lambda \in \cD$, $\Pi\in G(n,d)$ and
$X$ be a definable $\Lambda$-cone contained in  $\Pi$ and with origin $0$.  Then
$$\Theta_d(X)(0)=\int_{V\in \Omega\subset G(n,n-d)}
\Theta_d(p_V(X))(p_V(0))
\, d \mu_{n,d} .$$
\end{lem}

\begin{proof}

For every $V\in G(n,n-d)$, by linearity of
$p_V$, $p_V(X)$ is a $\Lambda$-cone
of $K^n/V$ with origin $p_V(0)$, and as
$\dim(\Pi)=\dim(K^n/V)$,
$p_V(X)$ is isomorphic to $X$,  for generic $V$ ($V\in \Omega=\Omega_\Pi$).
In what follows we denote $p_V(0)$ by $0$.
Take an integer $e>0$ such that $\pi_K^e\in\Lambda$, where we recall that $\pi_K$ is a uniformizer of $R$.
The sets $X$ and $p_V(X)$ being $\Lambda$-cones,
one has the following disjoint union relations
\[
 X=\displaystyle \coprod_{z\in \ZZ} \  \pi_K^{ze}\cdot ( \coprod_{i=0}^{e-1} X \cap S(0,i) ) ,
 \]

\[
\hbox{ and }\ \ p_V(X)=\displaystyle \coprod_{z\in \ZZ}
\  \pi_K^{ze}\cdot  (\coprod_{c=0}^{e-1}( p_V(X)) \cap S(0,c) ).
 \]
 It follows by the definition of $\Theta_d$ that
\begin{equation}\label{a} \Theta_d(X)(0)=\frac{(1-q^{-d})^{-1}}{e}
\sum_{i=0}^{e-1}q^{id} \mu_d(X \cap S(0,i)) \end{equation}
\begin{equation}\label{b} \hbox{ and }\ \  \Theta_d(p_V(X))(0)=\frac{(1-q^{-d})^{-1}}{e}
\sum_{c=0}^{e-1}q^{cd} \mu_d(p_V(X)\cap S(0,c)).
 \end{equation}
For each $i=0,\ldots,e-1$, let $C_i$ be
$$
\coprod_{z\in\ZZ} \pi_K^{ze} \cdot p_V(X\cap S(0,i)).
$$
One has $p_V(X) =  \coprod_i C_i$ by the linearity of $p_V$, and, the $C_i$ are definable since $X$ is a $\Lambda$-cone. Define the disjoint definable sets $ A_c^i$, for $i$ and $c$ going from $0$ to $e-1$, by
$$
A_c^i =  C_i \cap S(0,c).
$$
Clearly
$$
\coprod_{i=0}^{e-1}A_c^i = p_V(X) \cap S(0,c).
$$
Moreover, the sets  $\pi_K^{i-c}\cdot A_c^i$ are disjoint by linearity of $p_V$ and by bijectivity of $p_V$ on $\Pi$.
By the fact that
$$
q^{cd}\mu_d(A^i_c)=q^{id}\mu_d(\pi_K^{i-c}\cdot A_c^i),
$$
we obtain
$$
 \sum_{c=0}^{e-1}q^{cd} \mu_d(p_V(X)\cap S(0,c))
 =
\sum_{c=0}^{e-1}q^{cd} \sum_{i=0}^{e-1}
\mu_d(A_c^i) =
\sum_{i=0}^{e-1} q^{id}\mu_d(\coprod_{c=0}^{e-1}
\pi_K^{i-c}\cdot A_c^i)
$$
 \begin{equation}\label{c}
 =\sum_{i=0}^{e-1}q^{id}\mu_d(B_V^i),
 \end{equation}
where
$B^i_V:=\displaystyle\coprod_{c=0}^{e-1}
\pi_K^{i-c}\cdot A_c^i$.
Let us now consider
$\Phi_{V}: \Pi\setminus \{0\}\to (K^n/V)\setminus \{0\}$,
defined by
$$
\Phi_V( x)=
\pi_K^{  \ord(x) - \ord (p_V(x))  } \cdot  p_V(x).
$$
This map is bijective from
$X\cap S(0,i)$ to $B^i_V$, since $p_V$ is bijective from  $X$ to $p_V(X)$.
By change of variables one obtains
$$\mu_d(B_V^i)= \int_{X\cap S(0,i)}
\vert
\Jac
( \Phi_V) \vert \, d\mu_d.  $$
Furthermore, by Fubini,
\begin{equation}\label{d} \int_{V\in \Omega} \mu_d(B_V^i)\, d\mu_{n,d}=
\int_{x\in X\cap S(0,i)}\int_{V\in \Omega}
\vert
\Jac
( \Phi_V)(x) \vert \, d\mu_{n,d}(V) \, d\mu_d(x).\end{equation}
Note that,
for $x\in S(0,i)$, the quantity
$\kappa_i= \int_{V\in \Omega}
\vert
\Jac
( \Phi_V)(x) \vert \, d\mu_{n,d}(V)$ does not depend on $x$.
Indeed,
$\GL_n (R)$ acts transitively on
$S(0,i)$, $\mu_{n,d}$ is invariant under
this action and if $g\in \GL_n (R)$ and
$x'=g\cdot x$ for $x,x'\in S(0,i)$, then
$\Jac
(\Phi_V)(x')=
\Jac
(\Phi_{g^{-1}\cdot V})(x)$.
Moreover, by linearity of $p_V$, one has that $\kappa_i=\kappa$ is independent of $i$.
It follows from (\ref{b}), (\ref{c}) and (\ref{d}) that
$$  \int_{V\in \Omega} \Theta_d(p_V(X))(0)\, d\mu_{n,d}(V)
=\frac{(1-q^{-d})^{-1}}{e}\sum_{i=0}^{e-1}
 q^{id} \int_{x\in X\cap S(0,i)} \kappa\, d\mu_d
 (x).
 $$
$$= \kappa \cdot \frac{(1-q^{-d})^{-1}}{e}\sum_{i=0}^{e-1}    q^{id}   \mu_d(X\cap
S(0,i)).
$$
Finally, by (\ref{a}), we obtain
$$ \int_{V\in \Omega} \Theta_d(p_V(X))(0)\, d\mu_{n,d}(V)
=\kappa\cdot \Theta_d(X)(0).$$
One gets  $\kappa=1$ by taking  $X=\Pi$ in the latter formula.
\end{proof}

Lemma \ref{lccconeplan} may be viewed as
the tangential formulation of the
local Crofton formula for general definable $\Lambda$-cone sets
and its proof captures its geometric measure content.
Note that its proof still works
assuming that $X$ is a definable $\Lambda$-cone of dimension $d$
in $K^n$, instead of
a definable $\Lambda$-cone of dimension $d$ contained in some
$d$-dimensional vector space $\Pi$. Indeed, it is  essentially enough to
replace, in the proof
of Lemma \ref{lccconeplan},
$\Phi_V:\Pi\setminus \{0\}\to (K^n/V)\setminus \{0\}$ by the
restriction
of the mapping $ w\mapsto
\Psi_V(x)=  \pi_K^{  \ord(x) - \ord (p_V(x))  } p_V(x)$
on the smooth part of $X$
(the fibers of $\Psi_{V\vert X}$ being counted with multiplicity
in the area formula).
Hence we get the following extension of Lemma \ref{lccconeplan}:

\begin{lem}\label{lcccone}
 Let $\Lambda \in \cD$ and
$X$ be a definable $\Lambda$-cone of $K^n$ with origin $0$.
Then 
$$\Theta_d(X)(0)=\int_{\Omega\subset G(n,n-d)}
\Theta_d(p_{V!,0}(\11_X)) \
 d \mu_{n,d}(V) .
 \qed
 $$
\end{lem}

\begin{remark}
For $V\in G(n,n-d)$ and $y\in (K^n/ V)\setminus \{0\}$,
let us denote
by
 $\tilde V_y$ the fiber $\Psi_V^{-1}(\{y\})$ of
$\Psi_V: K^n\setminus V\to (K^n/ V)\setminus\{0\} $, where
$\Psi_V(x)= \pi_K^{  \ord(x) - \ord (p_V(x))  }p_V(x)$.
Note that $\tilde V_y\subset S(0, \ord y )\setminus V$ and
 $GL_n(R)$ acts on
$\tilde V=\{\tilde V_y;V\in G(n,n-d), y\in (K^n/V)\setminus \{0\}\}$.
For $X$ a 
 definable set of dimension $d$ in
$S(0,c)$,
where  $c\in \ZZ$, 
the statement of Lemma
\ref{lcccone} may be
reformulated as
$$ \mu_d(X)=\int_{V\in \Omega} \int_{ y\in S(0,c)\subset K^n/V}
\#(X\cap \tilde V_y) \ d\mu_d(y)\ d\mu_{n,d}(V).$$
Now note that the mapping $(V,y)\mapsto \tilde V_y$
defined from
$\{(V,y); V\in G(n,n-d), y\in K^n/V\} $
 to $\tilde V$ is one-to-one and that
the image of the Haar measure of $GL_n(R)$
under $g\mapsto g\cdot \tilde V_0$ (for $\tilde V_0$ fixed in
$\tilde V$)
gives a  $GL_n(R)$-invariant measure $\nu$ on $\tilde V$ such
that for
$E\subset \tilde V$, $E$ subanalytic say, we have
$$ \nu(E)=\int_{V\in \Omega} \int_{ y\in S(0,c)\subset K^n/V}
\11_{E}(\tilde V_y) \ d\mu_d(y)\ d\mu_{n,d}(V).$$
To obtain the above equality, it is enough to remark that
the right hand side  gives a function on
subsets of $\tilde V$ which pulls back on $Gl_n(R)$
as a Haar measure. With these notations we see that Lemma \ref{lcccone}
is nothing else than the classical spherical Crofton formula
for $X\cap S(0,c)$ (for a standard reference see \cite{Fe}, Theorem 3.2.48 and
note that the proof may be applied in our setting):

\begin{theorem}\label{scc}
Let $X$ be a definable set of $ S(0,0)\subset K^n$ of dimension $d$,
 then
$$\mu_d(X)= \int_{V\in \Omega} \int_{ y\in S(0,0)\subset K^n/V}
\#(X\cap \tilde V_y) \ d\mu_d(y)\ d\mu_{n,d}(V)$$
$$ = \int_{\tilde v\in \tilde V}  \#(X\cap \tilde v)
\ d\nu(\tilde v).$$

\end{theorem}
\end{remark}

For the general setting we will use the following auxiliary lemma.
\begin{lem}\label{partition}
Let $\Lambda$ be in $\cD$ and
let $X\subset K^n$ be a definable set of dimension $d$. Suppose that  $p:K^n\to K^d$ is  a coordinate projection which is injective on $X$. Then there exist definable sets $C_j$ of dimension $<d$ and a finite partition of $X$ into definable parts $X_j$ such that $p$ is injective on $C_0^{\Lambda}(X_j)\setminus C_j$ for each $j$.
\end{lem}
\begin{proof}
Since $C_0^{\Lambda}(X)\subset C_0^{K^\times}(X)$ for any $\Lambda$ in $\cD$, we may suppose that $\Lambda=K^\times$. We may also suppose that $0\in \overline{X}\setminus X$.
Partition $C_0^{K^\times }(X)$ into finitely many definable parts $B_j$ such that $p$ is injective on each set $B_j$. By linearity of $p$ we may suppose that each $B_j$ is a $K^\times$-cone. For each $j$ let $B_j'$ be the definable subset of $K^n$ consisting of the union of all lines $\ell\in K^n$ through $0$ such that the distance between $\ell\cap S(0,0)$ and $B_j\cap S(0,0)$ is strictly smaller than the distance between $\ell\cap S(0,0)$ and $B_i\cap S(0,0)$ for all $i\not = j$. Put $X_{j}^0:= X\cap B_j'$ for each $j$, and take a finite definable partition of $X$ into parts $X_j$ satisfying $X_{j}^0\subset X_j$ for each $j$.
 By construction $C_0^{K^\times }(X_j) = C_0^{K^\times }(X_{j}^0) \subset \overline B_j$. Let $C_j$ be $\overline{B_j}\setminus B_j$.
Then the $X_j$ and $C_j$ are as desired.
\end{proof}

We now prove Theorem \ref{lcc} in its general setting, that is to say,
for $X$ a given definable
subset of $K^n$ of dimension $d$ instead of some
definable $\Lambda$-cone of $K^n$ as in Lemma \ref{lcccone}.

\begin{proof}[Proof  of Theorem \ref{lcc}]

As in the proof of Theorem \ref{mt} we may
 assume that $\varphi=\11_X$.
Up to a finite partition of $X$ into definable parts we may suppose that $X$ is the graph of an $\varepsilon$-analytic map $U\subset K^d\to K^{n-d}$ as in Corollary \ref{trivcorSC}  and then it follows by this corollary that
$$
\Theta_d(X)(0)= \Theta_d( SC_0^{\Lambda} (X) )(0) = \Theta_d ( C_0^{\Lambda}(X) )(0).
$$
 For $C_0^{\Lambda}(X) $ we know that Theorem  \ref{lcc} holds by Lemma \ref{lcccone}, that is
$$
\Theta_d ( C_0^{\Lambda}(X) ) (0)
= \int_{V\in \Omega \subset G (n, n-d)}
\Theta_d (p_{V !, 0} (\11_{C_0^{\Lambda}(X)} ) \,
d \mu_{n, d}(V).
$$
We claim that, for generic $V$,
$$
\Theta_d (p_{V!,0}(\11_{C_0^{\Lambda}(X)})) = \Theta_d (p_{V!,0}(\11_X))
$$
which finishes the proof. We prove the claim as follows. Fix $V$. By Lemma \ref{partition} we can partition $X$ into finitely many definable parts $X_j$ (depending on $V$) such that $p_V$ is injective on $X_j$ and, up to a definable set of dimension $<d$, also on $C_0^{\Lambda}(X_j)$.  
By  additivity it is now enough to prove that
$$
\Theta_d (p_{V}(C_0^{\Lambda}(X_j)))(0) = \Theta_d (p_{V}(X_j))(0),
$$
which follows from Theorem \ref{mt} for open sets since
$p_{V}(C_0^{\Lambda}(X_j)) = C_0^{\Lambda}(p_{V}(X_j)).$
  \end{proof}

\bibliographystyle{amsplain}

\begin{thebibliography}{SGA}
\bibitem{bekka}
K. Bekka, \textit{Regular
stratification of subanalytic sets},
Bull. London Math. Soc. \textbf{25} (1993),   7--16.




\bibitem{pres}
R. Cluckers, \textit{Presburger sets and $p$-minimal fields},
J. Symbolic Logic \textbf{68} (2003),  153--162.


\bibitem{Ccell}
R. Cluckers, \textit{Analytic $p$-adic cell decomposition and
integrals}, Trans. Amer. Math. Soc. \textbf{356} (2004),
1489--1499.

\bibitem{Cexp} R. Cluckers, \textit{Multi-variate {I}gusa theory:
Decay rates of $p$-adic exponential sums},
 Int. Math. Res. Not.  \textbf{76} (2004), 4093--4108.


\bibitem{CCL} R.~Cluckers, G.~Comte and F.~Loeser,
\emph {Lipschitz continuity properties for $p$-adic
semi-algebraic and subanalytic functions}, to appear in GAFA,
arXiv:0904.3853.


\bibitem{CLR1} R.~Cluckers, L.~Lipshitz and Z.~Robinson,
\emph {Analytic cell decomposition and analytic motivic
integration}, Ann. Sci. \'{E}cole Norm. Sup. \textbf{39} (2006),
535--568.

\bibitem{CLip}
R. Cluckers, L. Lipshitz, \textit{Fields with analytic structure},
arXiv:0908.2376.



\bibitem{cons}
R. Cluckers, F. Loeser, \textit{Constructible motivic functions and
motivic integration},
Invent. Math. \textbf{173} (2008),  23--121.




\bibitem{cohen}P. J. Cohen,
\textit{Decision procedures for real and {$p$}-adic fields}, Comm.
Pure Appl. Math. \textbf{22} (1969), 131--151.


\bibitem{comte}
G. Comte,
\textit{\' Equisingularit\'e r\'eelle : nombres de Lelong et images polaires},
Ann. Sci. {\'E}cole Norm. Sup. \textbf{33} (2000), 757--788.


\bibitem{lk}
G. Comte, M. Merle, \textit{\'Equisingularit\'e r\'eelle II : invariants
locaux et conditions de r\'egularit\'e},
Ann. Sci. {\'E}cole Norm. Sup. \textbf{41} (2008), 1--48.


\bibitem{coliro}
G. Comte, J.-M. Lion, J.-P. Rolin,
\textit{Nature log-analytique du volume des sous-analytiques},
Illinois J. Math. \textbf{44} (2000), 884--888.


\bibitem{D84}
J. Denef, \textit{The rationality of the Poincar\'e series
associated to the $p$-adic points on a variety}, Invent. Math.
\textbf{77} (1984), 1--23.

\bibitem{D85}
J. Denef, \textit{On the evaluation of certain $p$-adic integrals},
S\'eminaire de th\'eorie des nombres, Paris 1983--84, 25--47, Progr.
Math., \textbf{59}, Birkh\"auser Boston, Boston, MA, 1985.


\bibitem{Dcell}
J. Denef, \textit{$p$-adic semi-algebraic sets and cell
decomposition}, J. Reine Angew. Math. \textbf{369} (1986),
154--166.



\bibitem{DvdD}
J. Denef, L. van den Dries, \textit{$p$-adic and real subanalytic
sets}, Ann. of Math. \textbf{128}, (1988), 79--138.



\bibitem{Draper}
R. Draper,  \textit{Intersection theory in analytic geometry},
Math. Ann. \textbf{180} 1969, 175--204.


\bibitem{DHM} L.~van~den~Dries, D.~Haskell and D.~Macpherson,
\emph{One-dimensional $p$-adic subanalytic sets}, J. London Math.
Soc.  \textbf{59}(1999), 1--20.


\bibitem{sd}
L. van den Dries, P. Scowcroft, \textit{On the structure of
semi-algebraic sets over $p$-adic fields}, J. Symbolic Logic
\textbf{53} (1988), 1138--1164.


\bibitem{vdDries}
{L. van den Dries}, \textit{Dimension of definable sets, algebraic
boundedness and {H}enselian fields}, Ann. Pure Appl. Logic
\textbf{45} (1989), 189--209.


\bibitem{Fe}
H. Federer, \textit{Geometric measure theory},
Grundlehren Math. Wiss.
\textbf{153}, Springer Verlag (1969).

\bibitem{Heifetz}
D. Heifetz, \textit{
$p$-adic oscillatory integrals and wave front sets},
Pacific J. Math. \textbf{116} (1985),  285--305.


\bibitem{HeMeSa}
J. P. Henry, M. Merle, C. Sabbah, \textit{
Sur la condition de Thom stricte pour un morphisme analytique complexe},
Ann. Sci. {\'E}cole Norm. Sup. \textbf{17}  (1984), no. 2, 227--268.

\bibitem{Hi}
H. Hironaka, \textit{
Stratification and flatness},
Real and complex singularities
(Proc. Ninth Nordic Summer School/NAVF Sympos. Math., Oslo, 1976),
 Sijthoff and Noordhoff, Alphen aan den Rijn, (1977),  199--265.


 \bibitem{Kuo}
 T. C. Kuo,
 \textit{The ratio test for analytic Whitney stratifications },
Proceedings of Liverpool Singularities  Symposium I,
Lecture Notes in Math., \textbf{192}, Springer, Berlin, (1971),
141--149.


\bibitem{KurWExp}
K. Kurdyka,
\textit{On a subanalytic stratification satisfying a Whitney property
with exponent 1}.
Real algebraic geometry (Rennes, 1991),
Lecture Notes in Math., \textbf{1524}, Springer, Berlin, (1992),
316--322.

\bibitem{KurPar}
K. Kurdyka, A. Parusi\'nski,
\textit{$w\sb f$-stratification of subanalytic functions and the
Lojasiewicz inequality}.
C. R. Acad. Sci. Paris S\'er. I Math., \textbf{318}, no. 2,
(1994),  129--133.


\bibitem{KPR}
K. Kurdyka, J.-B.  Poly, G.  Raby,
\textit{Moyennes des fonctions sous-analytiques, densit\'e, c\^one
tangent et tranches}.
Real analytic and algebraic geometry (Trento, 1988),
Lecture Notes in Math., \textbf{1420}, Springer, Berlin, (1990),
170--177.

\bibitem{KR}
K. Kurdyka, G. Raby,
\textit{Densit\'e des ensembles sous-analytiques},
Ann. Inst. Fourier \textbf{ 39} (1989), 753--771.




\bibitem{Lelong}
P. Lelong, \textit{
 Int\'egration sur un ensemble analytique complexe},
Bull. Soc. Math. France \textbf{85} (1957),  239--262.

\bibitem{Lio}
J.-M. Lion, \textit{Densit\'e des ensembles semi-pfaffiens},
Ann. Fac. Sci. Toulouse Math.  \textbf{7} (1998),  87--92.

\bibitem{Loi1}
T. L. Loi, \textit{Thom stratifications for functions definable
in o-minimal structures on $(\RR, +, \cdot)$},
C. R. Acad. Sci. Paris S\'er. I Math. \textbf{324} (1997),
no. 12, 1391--1394.


\bibitem{Loi2}
T. L. Loi, \textit{Verdier and strict
Thom stratifications in o-minimal structures},
Illinois J. Math. \textbf{42} (1998),
 no. 2, 347--356.

\bibitem{oesterle}
J. Oesterl{\'e}, \textit{R{\'e}duction modulo $p^{n}$ des
sous-ensembles analytiques ferm{\'e}s de ${\ZZ}^{N}_{p}$}, Invent.
Math. \textbf{66} (1982), 325--341.

\bibitem{Paw}
W. Paw\l ucki, \textit{Le th\'eor\`eme de Puiseux pour une application
sous-analytique},
Bull. Polish Acad. Sci. Math. \textbf{32} (1984), no. 9-10, 555--560.


\bibitem{serre}
J.-P. Serre, \textit{Quelques applications du th\'eor\`eme de
densit\'e de Chebotarev}, Inst. Hautes {\'E}tudes Sci. Publ.
Math. \textbf{54} (1981), 323--401.

\bibitem{Siu}
Y.T. Siu, \textit{Analyticity of sets associated to Lelong numbers and
the extension of closed positive currents},
Invent. Math. \textbf{27} (1974), 53--156.

\bibitem{Thie}
P. Thie,  \textit{The Lelong number of a point of a complex analytic set},
Math. Ann.  \textbf{172} (1967), 269--312.



\bibitem{Thom}
R. Thom, \textit{ Ensembles et morphismes stratifi\'es},
Bull. Amer. Math. Soc.  \textbf{75} (1969), 240--284

\bibitem{Verdier} J. L. Verdier,
\textit{Stratifications de Whitney et th\'eor\`eme de Bertini-Sard},
Invent. Math \textbf{36} (1976),
295--312.



\bibitem{Whitney} H. Whitney,  \textit{
Tangents to an analytic variety},
Ann. of Math. \textbf{81} (1965),  496--549.




\end{thebibliography}

\end{document}